\documentclass{article}

\usepackage{geometry}
\usepackage{amsmath,amssymb,amsthm,comment,enumitem, array}
\usepackage{color}
\usepackage{parskip}
\usepackage{parskip}
\usepackage{setspace}
\usepackage{graphicx}
\usepackage{mathtools}
\usepackage{subfig}
\usepackage{comment}
\usepackage{mathabx}
\usepackage{mathrsfs}
\usepackage{setspace}
\usepackage{bm}
\usepackage[normalem]{ulem}
\usepackage{cancel}
\usepackage{csquotes}
\usepackage{biblatex}
\usepackage[colorlinks=true, allcolors=blue]{hyperref}
\hypersetup{colorlinks, citecolor=red, filecolor=black, linkcolor=blue, urlcolor=blue}
\usepackage{pgfplots}
\pgfplotsset{compat=1.18}
\usepackage{graphics}
\usepackage{graphicx}
\usepackage{tikz}
\usepackage{circuitikz}
\usepackage{tcolorbox}
\usepackage{enumitem}
\usepackage{subfiles}
\usepackage{ytableau}
\usepackage{thmtools}
\usepackage{multicol}
\usepackage{lineno}

\newcommand\bb{\mathbb}

\newcommand\NN{\bb{N}}

\newtheorem{theorem}{Theorem}
\newtheorem{definition}{Definition}
\newtheorem{lemma}{Lemma}
\newtheorem{proposition}{Proposition}
\newtheorem{claim}{Claim}

\newtheorem{problem}{Problem}
\newtheorem{corollary}{Corollary}

\newtheorem{conjecture}{Conjecture}

\title{A sharp spectral extremal result for\\
general non-bipartite graphs}
\author{John Byrne\thanks{University of Delaware, Department of Mathematical Sciences, \texttt{jpbyrne@udel.edu}}}
\date{\today}

\begin{document}

\maketitle

\begin{abstract}
    For a graph family $\mathcal F$, let $\mathrm{ex}(n,\mathcal F)$ and $\mathrm{spex}(n,\mathcal F)$ denote the maximum number of edges and maximum spectral radius of an $n$-vertex $\mathcal F$-free graph, respectively, and let $\mathrm{EX}(n,\mathcal F)$ and $\mathrm{SPEX}(n,\mathcal F)$ denote the corresponding sets of extremal graphs. Wang, Kang, and Xue showed that if $r\ge 2$ and $\mathrm{ex}(n,F)=e(T_{n,r})+O(1)$ then $\mathrm{SPEX}(n,\mathcal F)\subseteq\mathrm{EX}(n,\mathcal F)$ for $n$ large enough. Fang, Tait, and Zhai extended this result by showing if $e(T_{n,r})\le\mathrm{ex}(n,\mathcal F)<e(T_{n,r})+\lfloor n/2r\rfloor$ then $\mathrm{SPEX}(n,\mathcal F)\subseteq\mathrm{EX}(n,\mathcal F)$ for $n$ large enough, and asked for the maximum constant $c(r)$ such that $\mathrm{ex}(n,\mathcal F)\le e(T_{n,r})+(c(r)-\varepsilon)n$ guarantees such containment. In this paper we determine $c(r)$ exactly for all $r\ge 3$.
\end{abstract}

\hspace{.3in} \textbf{Keywords:} Spectral Tur\'an problem, Tur\'an graph

\hspace{.3in} \textbf{2020 MSC:} 05C35, 05C50

\section{Introduction}

Let $\mathcal F$ be a family of graphs. For $n\in\mathbb{N}$, let $\mathrm{ex}(n,\mathcal F)$ ($\mathrm{spex}(n,\mathcal F)$) denote the maximum number of edges (spectral radius of adjacency matrix) of an $n$-vertex graph with no subgraph isomorphic to a graph in $\mathcal F$. Let $\mathrm{EX}(n,\mathcal F)$ ($\mathrm{SPEX}(n,\mathcal F)$) be the set of graphs attaining the maximum number of edges (spectral radius) among $\mathcal F$-free graphs. The problem of determining $\mathrm{ex}(n,\mathcal F)$ is called the \textit{Tur\'an problem} and the problem of determining $\mathrm{spex}(n,\mathcal F)$ is called the \textit{spectral Tur\'an problem} for $\mathcal F$. The latter problem, introduced systematically by Nikiforov \cite{nikiforov2010spectral}, has since grown quite popular; for an excellent survey on this topic we refer the reader to \cite{li2021survey}. In particular there have been many papers addressing the question of when we can guarantee that $\mathrm{SPEX}(n,\mathcal F)=\mathrm{EX}(n,\mathcal F)$. Some early results in this direction are the spectral Tur\'an theorem \cite{nikiforov2007bounds} which shows that $\mathrm{SPEX}(n,K_{r+1})=\mathrm{EX}(n,K_{r+1})$ and the work on odd cycles \cite{nikiforov2008spectral} which shows that $\mathrm{SPEX}(n,C_{2k+1})=\mathrm{EX}(n,C_{2k+1})$ for large enough $n$. However, this phenomenon does not always occur; for example $\mathrm{SPEX}(n,C_{2k})\cap\mathrm{EX}(n,C_{2k})=\emptyset$ and in fact the extremal families are very different structurally \cite{cioabua2022spectral2,nikiforov2010spectral,zhai2020spectral,zhai2012proof}. Recently Liu and Ning \cite{liu2023spectral} proposed the problem of determining all graphs $F$ for which $\mathrm{SPEX}(n,F)\subseteq\mathrm{EX}(n,F)$ for large enough $n$.

In this paper we ask for a condition on the value of $\mathrm{ex}(n,\mathcal F)$ which implies some relation between $\mathrm{EX}(n,\mathcal F)$ and $\mathrm{SPEX}(n,\mathcal F)$. In \cite{byrne2024general} this question was addressed when $\mathcal F$ contains a bipartite graph, so here we focus on the case where all graphs in $\mathcal F$ are non-bipartite. The first positive result in this direction was proved by Wang, Kang, and Xue \cite{wang2023conjecture}, answering a conjecture of Cioab\u a, Desai, and Tait \cite{cioabua2022spectral}. Recently, the result of \cite{wang2023conjecture} was extended by Fang, Tait, and Zhai \cite{fang2024turan}. Below, $\chi(\mathcal F)$ denotes $\min\{\chi(F):F\in\mathcal F)\}$.

\begin{theorem}[Fang, Tait, Zhai \cite{fang2024turan}] \label{non bipartite paper result}
    Let $n$ be sufficiently large and let $\mathcal F$ be a finite graph family satisfying $\chi(\mathcal F)=r+1\ge 3$ and $\mathrm{ex}(n,\mathcal F)<e(T_{n,r})+\left\lfloor\frac{n}{2r}\right\rfloor$. Then $\mathrm{SPEX}(n,\mathcal F)\subseteq\mathrm{EX}(n,\mathcal F)$.
\end{theorem}

Motivated by this result, the authors of \cite{fang2024turan} posed the following problem.

\begin{problem} \label{problem 1}
    Determine the maximum constant $c=c(r)$ such that for any finite graph family $\mathcal F$ with $\chi(\mathcal F)=r+1\ge 3$, for all positive $\varepsilon<c$ and sufficiently large $n$, $\mathrm{ex}(n,\mathcal F)\le e(T_{n,r})+(c-\varepsilon)n$ implies $\mathrm{SPEX}(n,\mathcal F)\subseteq\mathrm{EX}(n,\mathcal F)$.
\end{problem}

Thus, \autoref{non bipartite paper result} shows that $c(r)\ge 1/2r$ for every $r$. In this paper we determine $c(r)$ exactly for every $r\ge 3$. When $r=2$, we prove that $c(2)\le 1/3$ and conjecture that this is the correct value. These results also improve the upper bound $c(2)\le 3/2$ obtained by considering the odd wheel $W_{15}$ \cite{yuan2020extremal,cioabua2022spectral} (we are not aware of existing upper bounds for larger $r$).

The authors of \cite{fang2024turan} also asked the following.

\begin{problem} \label{problem 2}
    For $r\ge 2$ and $n$ sufficiently large, if a finite graph family $\mathcal F$ satisfies that $e(G)=e(T_{n,r})+\Theta(n)$ and $T_{n,r}\subseteq G$ for every $G\in\mathrm{EX}(n,\mathcal F)$, determine whether or not we have $\mathrm{SPEX}(n,\mathcal F)\subseteq\mathrm{EX}(n,\mathcal F)$.
\end{problem}

We answer this question in the negative for all $r\ge 3$ (see the proof of \autoref{counterexample 1}).

\subsection{Notation and definitions} \label{subsection notation}

We write $K_n$ for the complete graph on $n$ vertices; $K_{V_1,\ldots,V_r}$ for the complete multipartite graph with partite sets $V_1,\ldots,V_r$; $T_{n,r}$ for the $r$-partite Tur\'an graph on $n$ vertices; $P_\ell$ for a path on $\ell$ vertices; and $M_n$ for a matching on $n$ vertices (or almost matching if $n$ is odd). For a graph family $\mathcal F$, let $\chi(\mathcal F)=\min\{\chi(F):F\in\mathcal F\}$. Given a graph $G$, we write $V(G)$ and $E(G)$ for its vertex set and edge set, respectively; given a set $S\subseteq V(G)$, we write $N_G(S):=\{v\in V(G):v\sim u\text{ for all }u\in S\}$, or just $N(S)$ when the underlying graph is understood; $G[U]$ for the induced subgraph on $U$, $E[U]$ for $E(G[U])$ and $e(U)$ for $|E(U)|$; $G[U,W]$ for the bipartite subgraph induced by $U$ and $W$; $E(U,W)$ for the set of edges with one endpoint in $U$ and one endpoint in $W$ (where possibly $U\cap W\neq\emptyset$); $A(G)$ for the adjacency matrix of $G$ and $\lambda(G)$ for the spectral radius of $A(G)$. For a vertex $v\in V(G)$ and a set $S\subseteq V(G)$, we write $N_S(v):=N(v)\cap S$ and $d_S(v)$ for $|N_S(v)|$; we write $G-v$ for the graph obtained by deleting the vertex $v$; and $d'(v)$ for the number of walks of length two starting from $v$. For $n\in\NN$, let $U(G,n)$ be a graph on $n$ vertices consisting of as many disjoint copies of $G$ as possible with isolated vertices in the remainder. We write $G\cup H$ and $G+H$ for the disjoint union and join, respectively, of $G$ and $H$ and $k\cdot G$ for the union of $k$ disjoint copies of $G$.

\subsection{Main results} \label{section main results}

Our main result extends \autoref{non bipartite paper result} in the case $r\ge 3$. Let $\alpha$ be the largest root of the quadratic
$$f(x)=-\frac{1}{r-1}x^2+\left(2+\frac{5}{r-1}-\frac{4}{r}\right)x+\left(\frac{4}{r}-\frac{6}{r-1}\right),$$
and let $c_1(r)=\frac{\alpha-1}{\alpha r},$ i.e.
$$c_1(r)=\frac{1}{r}\left(1-\left[\frac{r-1}{2}\left(2+\frac{5}{r-1}-\frac{4}{r}+\sqrt{\left(2+\frac{5}{r-1}-\frac{4}{r}\right)^2-\frac{4}{r-1}\left(\frac{6}{r-1}-\frac{4}{r}\right)}\right)\right]^{-1}\right).$$

\begin{theorem} \label{main theorem}
Let $\mathcal F$ be a finite graph family satisfying $\chi(\mathcal F)=r+1\ge 4$ and $\mathrm{ex}(n,\mathcal F)\le e(T_{n,r})+Qn$ for large enough $n$, where $Q<c_1(r)$.
Then we have $\mathrm{SPEX}(n,\mathcal F)\subseteq\mathrm{EX}(n,\mathcal F)$ for large enough $n$.
\end{theorem}

The appearance of the constant $c_1(r)$ is explained as follows. The `worst case scenario' in the proof of \autoref{main theorem} occurs when an edge extremal graph $H$ is obtained from $T_{n,r}$ by embedding paths $P_k$ in a \textit{large} part along with one extra edge, while a spectral extremal graph $G$ is obtained from $T_{n,r}$ by embedding stars $K_{1,k-1}$ in a \textit{small} part. (The $k$-vertex trees with smallest and largest spectral radius are $P_k$ and $K_{1,k-1}$, respectively, so intuitively they should also give the extremal contributions to the spectral radius of a larger graph in which they are embedded.) We will show that
$$\lambda(H)-\lambda(G)=\left[2-\frac{k-5+6/k}{r-1}-\frac{4(k-1)}{kr}\right]\frac{1}{n}+O(n^{-2}).$$
The first term inside the brackets accounts for the extra edge; the second term comes from the difference between embedding stars versus paths; and the last term comes from the difference between embedding the trees in a small part versus a large part. In order to avoid this potential counterexample to \autoref{main theorem}, we require that the bracketed expression be positive, which is equivalent to $k<\alpha$. The number of edges embedded in the large part of $H$ is $\frac{k-1}{kr}n+O(1)<c_1(r)n$. Regarding the other possibility $k>\alpha$, we can always construct an $\mathcal F$ which gives rise to this `worst case scenario,' which proves that the constant $c_1(r)$ in \autoref{main theorem} is sharp up to some integrality considerations:

\begin{proposition} \label{counterexample 1}
    Suppose that $r\ge 3$, $k\in\NN$, and $k>\alpha$.
    Then there is a finite family of graphs $\mathcal F$ such that $\chi(F)=r+1$, $\mathrm{ex}(n,\mathcal F)\le e(T_{n,r})+\frac{k-1}{kr}n+O(1)$, and $\mathrm{SPEX}(n,\mathcal F)\cap\mathrm{EX}(n,\mathcal F)=\emptyset$ for an infinite sequence of $n$. Thus, $c(r)\le\frac{k-1}{kr}.$
\end{proposition}

Combining \autoref{main theorem} and \autoref{counterexample 1} allows us to determine $c(r)$ for $r\ge 3$:

\begin{corollary} \label{corollary} For every $r\ge 3$, we have
$$c(r)=\frac{\lceil\alpha\rceil -1}{\lceil\alpha\rceil r}.$$
\end{corollary}

As a consequence of \autoref{corollary}, we have $c(r)\sim1/r$ as $r\to\infty$.

The assumption in \autoref{main theorem} that $r\ge 3$ is required, as shown by the following result.

\begin{proposition} \label{counterexample 2}
    There exists a finite graph family $\mathcal F$ such that $\chi(F)=3$, $\mathrm{ex}(n,\mathcal F)\le e(T_{n,2})+\frac{1}{3}n$, and $\mathrm{SPEX}(n,\mathcal F)\cap\mathrm{EX}(n,\mathcal F)=\emptyset$ for an infinite sequence of $n$.
\end{proposition}

Given \autoref{non bipartite paper result} and \autoref{counterexample 2} one can show that $c(2)\in\{1/4,1/3\}$. We conjecture that $c(2)=1/3$:

\begin{conjecture} \label{conjecture}
    Let $n$ be sufficiently large and let $\mathcal F$ be a finite graph family satisfying $\chi(F)=3$ and $\mathrm{ex}(n,\mathcal F)\le e(T_{n,2})+Qn$, where $Q<\frac{1}{3}$. Then for $n$ large enough we have $\mathrm{SPEX}(n,\mathcal F)\subseteq\mathrm{EX}(n,\mathcal F)$.
\end{conjecture}

\subsection{Applications} \label{subsection applications}

\autoref{main theorem} unifies several new and existing results. Taking $Q=\frac{1}{2r}$ in \autoref{main theorem}, one may obtain the case $r\ge 3$ of the previous results \cite{wang2023conjecture} and \cite{fang2024turan}. Using this, one can prove that $\mathrm{SPEX}(n,F)\subseteq\mathrm{EX}(n,F)$ for $n$ large enough when $F$ is a $(k,r)$\textit{-fan} obtained by intersecting $k$ copies of $K_r$ in a single vertex (where $r\ge 4$) \cite{cioaba2019spectral,desai2022spectral}. 

To obtain new applications of \autoref{main theorem}, we are interested in families $\mathcal F$ such that $e(T_{n,r})+\left\lfloor\frac{n}{2r}\right\rfloor\le\mathrm{ex}(n,\mathcal F)\le e(T_{n,r})+Qn$. One example is the graph $P_{2p+1}^p$ obtained by joining all vertices at distance at most $p$ in a path of length $2p+1$, where $p\ge 3$. Then Theorem 1.2 of \cite{yuan2022extremal} gives that $\mathrm{EX}(n,\mathcal F)$ contains a graph obtained from $T_{n,p}$ with a matching embedded in one part, hence \autoref{main theorem} implies that $\mathrm{SPEX}(n,P_{2p+1}^p)\subseteq\mathrm{EX}(n,P_{2p+1}^p)$ for $n$ large enough.

We are not aware of any other existing extremal results in this category. However, with the goal of constructing graphs $F$ for which we can show $\mathrm{ex}(n,F)=e(T_{n,r})+\frac{n}{2r}+O(1)$ we began exploring specific small graphs. Other than the path power discussed above, the simplest such $F$ we found is described as follows. The graph $F_1$ consists of two paths on four vertices $bcde$, $fghi$, with additional edges $bf,bg,cf,ch,dg,di,eh,ei$ and additionally one dominating vertex $a$.

\begin{figure}[!ht]
\centering
\resizebox{0.5\textwidth}{!}{%
\begin{circuitikz}
\tikzstyle{every node}=[font=\Large]
\draw [ fill={rgb,255:red,0; green,0; blue,0} ] (11.25,10.25) circle (0.25cm);
\draw [ fill={rgb,255:red,0; green,0; blue,0} ] (8.75,10.25) circle (0.25cm);
\draw [ fill={rgb,255:red,0; green,0; blue,0} ] (8.75,12.75) circle (0.25cm);
\draw [ fill={rgb,255:red,0; green,0; blue,0} ] (13.75,12.75) circle (0.25cm);
\draw [ fill={rgb,255:red,0; green,0; blue,0} ] (11.25,12.75) circle (0.25cm);
\draw [ fill={rgb,255:red,0; green,0; blue,0} ] (13.75,10.25) circle (0.25cm);
\node [font=\LARGE] at (13,9.75) {};
\node [font=\LARGE] at (13,9.75) {};
\node [font=\LARGE] at (13.5,9.25) {};
\draw [ fill={rgb,255:red,0; green,0; blue,0} , line width=1pt ] (12.5,15.25) circle (0.25cm);
\draw [ fill={rgb,255:red,0; green,0; blue,0} , line width=1pt ] (16.25,10.25) circle (0.25cm);
\draw [ fill={rgb,255:red,0; green,0; blue,0} , line width=1pt ] (16.25,12.75) circle (0.25cm);
\draw [line width=1pt, short] (8.75,12.75) -- (12.5,15.25);
\draw [line width=1pt, short] (12.5,15.25) -- (11.25,12.75);
\draw [line width=1pt, short] (12.5,15.25) -- (8.75,10.25);
\draw [line width=1pt, short] (12.5,15.25) -- (11.25,10.25);
\draw [line width=1pt, short] (12.5,15.25) -- (13.75,10.25);
\draw [line width=1pt, short] (12.5,15.25) -- (13.75,12.75);
\draw [line width=1pt, short] (12.5,15.25) -- (16.25,12.75);
\draw [line width=1pt, short] (12.5,15.25) -- (16.25,10.25);
\draw [line width=1pt, short] (8.75,12.75) -- (8.75,10.25);
\draw [line width=1pt, short] (8.75,10.25) -- (11.25,12.75);
\draw [line width=1pt, short] (8.75,12.75) -- (11.25,12.75);
\draw [line width=1pt, short] (8.75,12.75) -- (11.25,10.25);
\draw [line width=1pt, short] (8.75,10.25) -- (11.25,10.25);
\draw [line width=1pt, short] (11.25,12.75) -- (13.75,12.75);
\draw [line width=1pt, short] (11.25,10.25) -- (13.75,10.25);
\draw [line width=1pt, short] (11.25,12.75) -- (13.75,10.25);
\draw [line width=1pt, short] (11.25,10.25) -- (13.75,12.75);
\draw [line width=1pt, short] (13.75,12.75) -- (16.25,12.75);
\draw [line width=1pt, short] (16.25,12.75) -- (16.25,10.25);
\draw [line width=1pt, short] (16.25,10.25) -- (13.75,10.25);
\draw [line width=1pt, short] (13.75,12.75) -- (16.25,10.25);
\draw [line width=1pt, short] (13.75,10.25) -- (16.25,12.75);
\node [font=\Large] at (12.5,15.75) {\textit{a}};
\node [font=\Large] at (8.75,13.25) {\textit{b}};
\node [font=\Large] at (11.25,12.25) {\textit{c}};
\node [font=\Large] at (13.75,12.25) {\textit{d}};
\node [font=\Large] at (16.25,13.25) {\textit{e}};
\node [font=\Large] at (8.75,9.75) {\textit{f}};
\node [font=\Large] at (11.25,9.75) {\textit{g}};
\node [font=\Large] at (13.75,9.75) {\textit{h}};
\node [font=\Large] at (16.25,9.75) {\textit{i}};
\end{circuitikz}
}%
\caption{The graph $F_1$.}
\label{fig:my_label}
\end{figure}
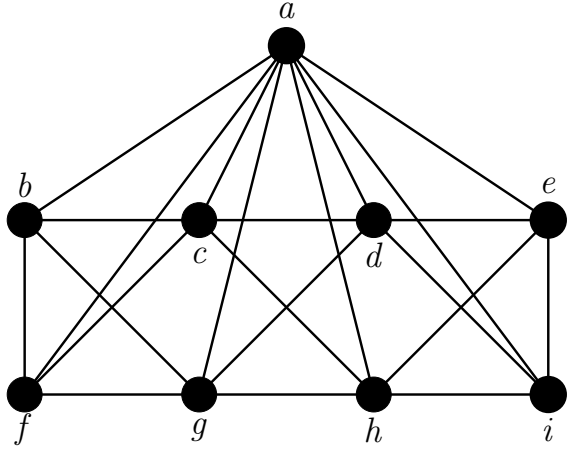

\begin{claim} \label{F1 properties}
We have the following properties of $F_1$.
\begin{itemize}
    \item[(i)] $\chi(F_1)=4$.
    \item[(ii)] The graph obtained from $T_{n,3}$ by adding an edge to two different parts contains a copy of $F_1$.
    \item[(iii)] The graph obtained from $T_{n,3}$ by embedding $P_3\cup P_2$ in part contains a copy of $F_1$.
    \item[(iv)] The graph obtained from $T_{n,3}$ by embedding a matching in one part does not contain a copy of $F_1$.
\end{itemize}
\end{claim}
\begin{proof}
    \begin{itemize}
        \item[(i)] One can check that $\chi(F_1-a)=3$ so $\chi(F_1)=4$.
        \item[(ii)] If the added edges are $u_1v_2$ and $v_1v_2$ then we find a copy of $F_1$ by identifying $bf$ with $u_1u_2$ and $ei$ with $v_1v_2$.
        \item[(iii)] If the added paths are $u_1u_2u_3$ and $v_1v_2$ then we find a copy of $F_1$ by identifying $gbc$ with $u_1u_2u_3$ and $ei$ with $v_1v_2$. 
        \item[(iv)] Suppose $F_1$ appears in this graph. Assume without loss of generality that the matching is embedded in $V_1$, $a\in V_3$, and $b\in V_1$. If $f\in V_1$ then $c,g\in V_2$ so $d,h\in V_1$ and so $e,i\in V_2$, contradicting $e\sim i$. If $f\in V_2$ then $c,g\in V_1$ so $gbc$ is a $P_3$ in $V_1$, a contradiction.
    \end{itemize}
\end{proof}

Facts (i), (ii), and (iii) give that the assumptions of \autoref{structure lemma 1} below are satisfied, hence the conclusion of \autoref{structure lemma 1} and (iv) give that $\mathrm{EX}(n,F_1)$ is obtained by embedding a matching in one part of $T_{n,3}$. This gives a value of $Q=1/2r$ in \autoref{main theorem}, and we conclude that $\mathrm{SPEX}(n,F_1)=\mathrm{EX}(n,F_1)$.

\section{Proof of \autoref{main theorem}}

Our overall strategy is as follows. In \autoref{section structure of edge extremal} we prove \autoref{structure lemma 1}, which provides general structural information about the graphs in $\mathrm{EX}(n,\mathcal F)$. In \autoref{section strucutre of spectral extremal}, we prove that the same general structural description also applies to graphs in $\mathrm{SPEX}(n,\mathcal F)$. The difference in spectral radius of two graphs $H,G$ satisfying this structural description is equal to $2(e(H)-e(G))/n+O(n^{-2})$, plus an error term which depends on what trees are embedded in $H$ and $G$ and in which part of their $r$-partition. We show that this error term is at most $\left[-\frac{k-5+6/k}{r-1}-\frac{4(k-1)}{kr}\right]\frac{1}{n}+O(n^{-2})$. If $e(G)=e(H)$ then $\mathrm{SPEX}(n,\mathcal F)\subseteq\mathrm{EX}(n,\mathcal F)$, while otherwise $e(G)<e(H)$ and we have
$$\lambda(H)-\lambda(G)\ge\left[2-\frac{k-5+6/k}{r-1}-\frac{4(k-1)}{kr}\right]\frac{1}{n}+O(n^{-2}).$$
The assumption $Q<c_1(r)$ in \autoref{main theorem} guarantees that the bracketed term is positive, giving a contradiction.

\begin{definition} \label{k value}
Given a graph family $\mathcal F$ with $\chi(F)=r+1\ge 3$ and $\mathrm{ex}(n,\mathcal F)\le e(T_{n,r})+Qn$ for large enough $n$ where $Q<1/r$, let $k=k(\mathcal F)$ be the maximum integer such that for some tree $T_k$ of order $k$, for all $m$, $U(T_k,m)+T_{m(r-1),r-1}$ is $\mathcal F$-free. 
\end{definition}

Since $Q<1/r$, $k$ is finite; more precisely, whenever $k|m$ we have $m\frac{k-1}{k}\le Qmr$, hence $k\le\lfloor 1/(1-Qr)\rfloor$.

\subsection{Structure of edge extremal graphs} \label{section structure of edge extremal}

In this subsection we prove the following lemma. Note that if $\mathcal F$ is as in \autoref{main theorem}, then $\mathcal F$ satisfies the assumptions of \autoref{structure lemma 1}. Our reason for using this lemma which is slightly more general than what is needed to prove \autoref{main theorem} is that it also helps us quickly determine the extremal graphs for the forbidden subgraphs appearing in our applications and counterexamples.

\begin{lemma} \label{structure lemma 1}
    Suppose $\mathcal F$ is a finite family of graphs, $\chi(\mathcal F)=r+1\ge 3$, and for $n$ large enough, any graph obtained by adding more than $Qnr$ edges to $T_{nr,r}$ is not $\mathcal F$-free, where $Q<1/r$. Then for $n$ large enough, $\mathrm{EX}(n,\mathcal F)$ contains a graph $G$ obtained from $T_{n,r}$ by embedding disjoint trees of some fixed order $k$ in one part, then making a bounded number of edge deletions and additions to the whole graph.
\end{lemma}

Let $k=k(\mathcal F)$ as in \autoref{k value}.

\begin{lemma} \label{bad graphs}
    For some large enough constant $m$, the following graphs are not $\mathcal F$-free:
    \begin{itemize}
        \item $K_{1,m}+T_{m(r-1),r-1}$
        \item $M_m+M_m+T_{m(r-2),r-2}$
        \item $(T_1\cup\cdots\cup T_m)+T_{m(r-1),r-1}$, for any trees $T_i$ of order $k+1$
        \item $(C_1\cup\cdots\cup C_m)+T_{m(r-1),r-1}$, for any cycles $C_i$ of order at most $k$
    \end{itemize}
\end{lemma}
\begin{proof}
    Assume for a contradiction one of the graphs above is $\mathcal F$-free.
    \begin{itemize}
        \item $K_{1,m}+T_{m(r-1),r-1}$ contains a graph obtained by adding $m-1$ edges to $T_{m(r-1),r-1}$, so $Q(m-1)r\ge m$ by assumption and $Q\ge\left(1-\frac{1}{m}\right)/r$, which contradicts $Q<1/r$ for large enough $m$.
        \item $M_m+M_m+T_{m(r-2),(r-2)}$ is obtained by adding $m$ edges to $T_{mr,r}$, so $Qmr\ge m$, hence $Q\ge 1/r$, a contradiction.
        \item By applying Cayley's formula and the pigeonhole principle, $(T_1\cup\cdots\cup T_m)+T_{m(r-1),r-1}$ contains a graph obtained by embedding $U(T,m/(k+1)^k)$ in one part of $T_{mr/(k+1)^{k},r}$ for some tree $T$ of order $k+1$. If $m$ is large enough, this violates the definition of $k(\mathcal F)$.
        \item Let $m'$ be the sum of the orders of the $C_i$, then $m'\ge m$ and $(C_1\cup\cdots\cup C_m)+T_{m'(r-1),r-1}$ is obtained from $T_{m'r,r}$ by adding $m'$ edges to one part, which contradicts $Q<1/r$.
    \end{itemize}
\end{proof}

Assume that $G=G_n$ is any sequence of graphs in $\mathrm{EX}(n,\mathcal F)$. Let $m$ be a large enough even constant satisfying \autoref{bad graphs}. Let $M=\max\{|V(F)|:F\in\mathcal F\}$. Choose constants $\varepsilon,\varepsilon_1,\varepsilon_2,\varepsilon_3$ satisfying the following inequalities (we do not attempt to avoid redundancy or optimize the constants).
\begin{center}
\begin{tabular}{p{0.3\textwidth}p{0.3\textwidth}}
\begin{itemize}
    \item[(i)] $M\sqrt\varepsilon<1/r$
    \item[(ii)] $\sqrt\varepsilon+2\varepsilon_1<1/r$
    \item[(iii)] $\sqrt\varepsilon+\varepsilon_1<\varepsilon_2$
    \item[(iv)] $\sqrt\varepsilon<\varepsilon_1$
    \item[(v)] $(2mr+1)\sqrt\varepsilon<\varepsilon_1$
\end{itemize} &
\begin{itemize}
    \item[(vi)] $\varepsilon_2+2\sqrt\varepsilon<\varepsilon_3$
    \item[(vii)] $\sqrt\varepsilon+m(k+1)\varepsilon_3<1/r$
    \item[(viii)] $\sqrt\varepsilon+m(k+r)\varepsilon_3<1/r$
    \item[(ix)] $\sqrt\varepsilon+M\varepsilon_3<1/r$
    \item[(x)] $\sqrt\varepsilon+2mr\varepsilon_3<1/r$
\end{itemize}
\end{tabular}
\end{center}
This is possible, as seen by choosing them in the order $\varepsilon_3,\varepsilon_2,\varepsilon_1,\varepsilon$.

We will rely on the following stability theorem of Erd\H{o}s and Simonovits.

\begin{lemma}[Erd\H{o}s \cite{erdos1967some,erdos1966some}, Simonovits \cite{simonovits1968method}] \label{strong version stability}
Suppose that $\mathcal F$ is a finite family of graphs with $\chi(\mathcal F)=r+1.$ For all $\varepsilon>0$ there exists $\delta>0$ such that the following holds: if $n$ is large enough, then any $n$-vertex $\mathcal F$-free graph $G$ with more than $e(T_{n,r})-\delta n^2$ edges can be obtained from $T_{n,r}$ by deleting and adding at most $\varepsilon n^2$ edges.
\end{lemma}

With respect to our graph $G\in\mathrm{EX}(n,\mathcal F)$, and the constant $\varepsilon$ chosen above, this implies that $G$ can be obtained from $T_{n,r}$ by adding and deleting at most $\varepsilon n^2$ edges. Let $V(G)=V=V_1\sqcup\cdots\sqcup V_r$ where $|V_i|\in\{\lfloor n/r\rfloor,\lceil n/r\rceil\}$, and $G$ is obtained from $K_{V_1,\ldots,V_r}$ by adding a graph $A$ and removing a graph $B$ with $e(A)+e(B)\le\varepsilon n^2$. Let $X=\{v\in V:d_B(x)\ge 2\sqrt\varepsilon n\}$, and note that we have
$$2|X|\sqrt{\varepsilon}n\le\sum_{v\in X}d_B(v)= 2e(B)\le 2\varepsilon n^2\Longrightarrow |X|\le\sqrt{\varepsilon}n.$$

For $v\in V_i$, we say that $v$ is $V$-\textit{ordinary} if $N(v)=\bigcup_{j\ne i}V_j-X$.

\begin{lemma} \label{ordinary vertex replacement}
    If $v\in X$, then there exists some $i$ such that $d_{V_i}(v)\le\varepsilon_1n$ and $d_{V_j}(v)\ge n/r-\varepsilon_2n$ for every $j\ne i$.
\end{lemma}
\begin{proof}
    Suppose not. Note than a $V$-ordinary vertex $u$ satisfies $d(u)\ge(1-1/r-\sqrt\varepsilon)n+O(1)$. 
    \begin{claim} \label{ordinary vertex replacement claim 1}
        If $G'$ is obtained from $G$ by deleting some vertex and adding a $V$-ordinary vertex $u\in V_i$, then $G'$ is $\mathcal F$-free.
    \end{claim}
    \begin{proof}
        Suppose $G'$ contains some $F\in\mathcal F$. Then $u\in V(F)$, and $|N_F(u)|\le M$, $N_F(u)\subseteq V-V_i-X$, so 
        $$|N_G(N_F(u))|\ge(1/r-M\sqrt\varepsilon)n+O(1)>0$$
        since $M\sqrt\varepsilon<1/r$, (i). Thus replacing $u$ by a vertex in $N_G(N_F(u))$ gives a copy of $F$ in $G$, a contradiction.
    \end{proof}
    If there are two different indices $i,i'$ such that $d_{V_i}(v),d_{V_{i'}}(v)\le\varepsilon_1n$ then we have
    $$d(v)\le(1-2/r+2\varepsilon_1)n+O(1).$$
    If we delete $v$ and add a $V$-ordinary vertex then we obtain a graph $G'$ such that
    $$e(G')-e(G)=(1/r-\sqrt\varepsilon-2\varepsilon_1)n+O(1)>0$$
    since $\sqrt\varepsilon+2\varepsilon_1<1/r$ (ii), which contradicts $G\in\mathrm{EX}(n,\mathcal F)$ by \autoref{ordinary vertex replacement claim 1}. If there is a unique index $i$ such that $d_{V_i}(v)\le\varepsilon_1n$ then by assumption there exists some $j\ne i$ such that $d_{V_j}(v)\le n/r-\varepsilon_2n$. Thus
    $$d(v)\le(1-1/r+\varepsilon_1-\varepsilon_2)n+O(1)$$
    and if we delete $v$ and add an $V$-ordinary vertex then we obtain a graph $G'$ such that
    $$e(G')-e(G)=(\varepsilon_2-\varepsilon_1-\sqrt\varepsilon)n+O(1)>0$$
    since $\sqrt\varepsilon+\varepsilon_1<\varepsilon_2$ (iii), which contradicts $G\in\mathrm{EX}(n,\mathcal F)$ by \autoref{ordinary vertex replacement claim 1}. Finally, suppose there is no index $i$ for which $d_{V_i}(v)\le\varepsilon_1n$, and assume without loss of generality $v\in V_1$. Now
    $$|N(v)\cap(V_1-X)|\ge(\varepsilon_1-\sqrt\varepsilon)n+O(1)\ge m$$
    since $\sqrt\varepsilon<\varepsilon_1$ (iv), so there are some $m$ vertices $v_{11},\ldots,v_{1m}\in N(v)\cap(V_1-X)$. Having found $v_{11},\ldots,v_{1m},$ $\ldots,v_{i1},\ldots,v_{im}$, we have
    $$|N(\{v,v_{11},\ldots,v_{im}\})\cap V_{i+1}-X|\ge(\varepsilon_1-2im\sqrt\varepsilon-\sqrt\varepsilon)n+O(1)\ge m$$
    since $(2mr+1)\sqrt\varepsilon<\varepsilon_1$ (v), so we can continue this process until we have selected $v_{11},\ldots,v_{rm}$. Then $G[\{v,v_{11},\ldots,v_{rm}\}]\supseteq K_{1,m}+T_{m(r-1),r-1}$, contradicting \autoref{bad graphs}.
\end{proof}

By \autoref{ordinary vertex replacement}, we may define a new partition $V=W_1\sqcup\cdots\sqcup W_r$ as follows. If $v\in V_i-X$ then let $v\in W_i$. If $v\in X$ then let $v\in W_i$ where $i$ is chosen such that $d_{V_i}(v)\le\varepsilon_1 n$ and $d_{V_j}(v)\ge n/r-\varepsilon_2n$ for every $j\ne i$. Let $w_i:=|W_i|$. Note that $w_i\in(n/r-\sqrt\varepsilon n+O(1),n/r+\sqrt\varepsilon n+O(1))$. Thus for each $v\in W_i$, for each $j\ne i$, $d_{W_j}(v)\ge w_j-(\varepsilon_2+2\sqrt\varepsilon)n+O(1)\ge w_j-\varepsilon_3n$, since $\varepsilon_2+2\sqrt\varepsilon<\varepsilon_3$, (vi). Recall that $m$ is a large even integer satisfying \autoref{bad graphs}.

\begin{lemma} \label{automatic joining}
    Let $K$ be a graph of order at most $m(k+1)$. For $n$ large enough, if $G[W_i]\supseteq K$ for some $i$ then $G\supseteq K+T_{m(r-1),r-1}$.
\end{lemma}
\begin{proof}
    Assume without loss of generality $i=1$, then
    $$|N(V(K))\cap W_2|\ge w_2-|V(K)|\varepsilon_3n+O(1)\ge(1/r-\sqrt\varepsilon-m(k+1)\varepsilon_3)n+O(1)\ge m$$
    since $\sqrt\varepsilon+m(k+1)\varepsilon_3<1/r$, (vii), so assume $v_{21},\ldots,v_{2m}\in N(V(K))\cap W_2$. Having found $v_{21},\ldots,v_{im}$, we have
    $$\begin{aligned}|N(V(K)\cup\{v_{21},\ldots,v_{im}\})\cap W_{i+1}|&\ge w_{i+1}-(|V(K)|+(i-1)m)\varepsilon_3n+O(1)\\
    &\ge(1/r-\sqrt\varepsilon-m(k+r)\varepsilon_3)n+O(1)\ge m
    \end{aligned}$$
    since $\sqrt\varepsilon+m(k+r)\varepsilon_3<1/r$, (viii). Thus, we can continue the process until we have found $v_{21},\ldots,v_{rm}$, and $G[V(K)\cup\{v_{21},\ldots,v_{rm}\}]\supseteq K+T_{m(r-1),r-1}$.
\end{proof}

Call a vertex $v\in W_i$ $W$-\textit{ordinary} if $N(v)=V-W_i$.

\begin{lemma} \label{ordinary vertex replacement 2}
    Let $G'$ be obtained from $G$ by deleting some vertex and adding a $W$-ordinary vertex $u\in W_i$. Then $G'$ is $\mathcal F$-free.
\end{lemma}
\begin{proof}
    Suppose $G'$ contains some $F\in\mathcal F$. Then $u\in V(F)$, $|N_F(u)|\le M$ and $N_F(u)\subseteq V-W_i$. Hence
    $$|N_G(N_F(u))|\ge w_i-M\varepsilon_3n+O(1)\ge(1/r-\sqrt\varepsilon-M\varepsilon_3)n+O(1)>0$$
    since $\sqrt\varepsilon+M\varepsilon_3<1/r$, (ix), so replacing $u$ by some vertex in $N_G(N_F(u))$ gives a copy of $F$ in $G$, a contradiction.
\end{proof}

\begin{lemma} \label{easy improvement}
    For each $W_i$, for each $v\in W_i$, we have $d_{W_i}(v)\le m$.
\end{lemma}
\begin{proof}
    This follows immediately from \autoref{automatic joining} and \autoref{bad graphs}.
\end{proof}

\begin{lemma} \label{easy improvement 2}
    For each $v\in W_i$, for each $j\ne i$, we have $d_{W_j}(v)\ge w_j-m$.
\end{lemma}
\begin{proof}
    If $d_{W_j}(v)\le w_j-m-1$ then by \autoref{easy improvement} replacing $v$ with a $W$-ordinary vertex increases the number of edges, and by \autoref{ordinary vertex replacement 2} and $G\in\mathrm{EX}(n,\mathcal F$) this is a contradiction.
\end{proof}

\begin{lemma} \label{strong one part}
    For some constant $C$, there is at most one $i$ such that $e(W_i)\ge C$.
\end{lemma}
\begin{proof}
    Assume for a contradiction that $e(W_1),e(W_2)$ are both arbitrarily large. By \autoref{easy improvement}, $G[W_1],G[W_2]$ have bounded maximum degree, so this implies we can find a matching $M_1$ in $W_1$ of order $m$ and a matching $M_2$ in $W_2$ of order $2(m+1)m$. For each $v\in M_1$, by \autoref{easy improvement 2}, there are at most $m$ edges in $M_2$ either of whose endpoints are not adjacent to $v$, which implies that there are some $(m+1)m-m^2=m$ edges in $M_2$ inducing a matching $M_2'$ such that $G[M_1,M_2']\simeq K_{m,m}$. Similarly to lemmas proved above, having found $v_{31},\ldots,v_{im}$, we have
    $$\begin{aligned}|N(M_1\cup M_2'\cup\{v_{31},\ldots,v_{im}\})\cap W_{i+1}|&\ge w_{i+1}-(2m+(i-2)m)\varepsilon_3n+O(1)\\
    &\ge(1/r-\sqrt\varepsilon-2mr\varepsilon_3)n+O(1)\ge m
    \end{aligned}$$
    since $\sqrt\varepsilon+2mr\varepsilon_3<1/r$, (x). Thus, we find $M_1,M_2',v_{31},\ldots,v_{rm}$ whose union induces an $M_m+M_m+T_{(m-2)r,r}$, which contradicts \autoref{bad graphs}.
\end{proof}

Assume without loss of generality that $W_1$ is the part for which $e(W_1)\to\infty$ is possible, i.e. for all $i\ne 1$, we have $e(W_i)=O(1)$. For each $i\ne 1$, we may write $W_i=W_i^o\sqcup W_i^m$, where $e(W_i^o)=0$ and $|W_i^m|=O(1)$. 

\begin{lemma} \label{ordinary vertices}
    For each $i\ne 1$, all vertices in $W_i^o$ are $W$-ordinary.    
\end{lemma}
\begin{proof}
    Let $v\in W_i^o$ not be $W$-ordinary. Since $d_{W_i}(v)=0$, replacing $v$ with a $W$-ordinary vertex increases the number of edges, which by \autoref{ordinary vertex replacement 2} and $G\in\mathrm{EX}(n,\mathcal F)$ is a contradiction.
\end{proof}

\begin{lemma} \label{structure of w1}
    We can write $W_1=W_1^o\sqcup W_1^m$, where $G[W_1^o]$ is a disjoint union of trees of order $k$, each tree component of $G[W_1^o]$ occurs more than $M(m+1)$ times, and $|W_1^m|=O(1)$.
\end{lemma}
\begin{proof}
    We prove the result in a series of a claims.
    \begin{claim} \label{structure of w1 claim 1}
        The number of components of $G[W_1]$ whose order is at least $k+1$ is at most $m$.
    \end{claim}
    \begin{proof}
        If not, then we can find trees $T_1,\ldots,T_m$ of order $k+1$ such that $T_1\cup\cdots\cup T_m\subseteq G[W_1]$. Then by \autoref{automatic joining}, we have $(T_1\cup\cdots\cup T_m)+T_{m(r-1),r-1}\subseteq G$, contradicting \autoref{bad graphs}.
    \end{proof}
    \begin{claim} \label{structure of w1 claim 2}
        The graph $G[W_1]$ contains at most $m$ disjoint cycles of order at most $k$.
        \begin{proof}
            If not, then let the $m$ cycles be called $C_1,\ldots,C_m$. By \autoref{automatic joining}, we have $(C_1\cup\cdots\cup C_m)+T_{m(r-1),r-1}\subseteq G$, contradicting \autoref{bad graphs}.
        \end{proof}
    \end{claim}
    \begin{claim} \label{structure of w1 claim 3}
        The graph $G[W_1]$ has at least $\sqrt n$ isomorphic tree components of order $k$.
    \end{claim}
    \begin{proof}
        Suppose not. Then by \autoref{structure of w1 claim 1}, \autoref{structure of w1 claim 2}, and Cayley's formula, all but $k^{k-2}\sqrt n+O(1)$ of the components of $G[W_1]$ are smaller trees. By the definition of $k(\mathcal F)$, it is easy to check that the graph $G'$ obtained by deleting $E(W_i)$ for all $i$ and then embedding $U(T_k,w_1)$ in $W_1$ for some $k$-vertex tree $T_k$ is $\mathcal F$-free. Then we have
        $$e(G')-e(G)\ge w_1\frac{k-1}{k}+O(1)-k^{k-2}\sqrt n-w_1\frac{k-2}{k-1}>0$$
        which contradicts $G\in\mathrm{EX}(n,\mathcal F)$.
    \end{proof}
    Let $T'$ be a tree occurring at least $\sqrt n$ times as guaranteed by \autoref{structure of w1 claim 3}.
    \begin{claim} \label{structure of w1 claim 4}
        For any set $S\subseteq W_1$ with $|S|=k$, the graph $G'$ obtained by deleting $E(S,W_1)$ and adding a copy of $T'$ on $S$ is $\mathcal F$-free.
    \end{claim}
    \begin{proof}
        Assume for a contradiction that $G'\supseteq F\in\mathcal F$. Then $S':=V(F)\cap S\ne\emptyset$. Note there are at least $\sqrt n/2$ copies of $T'$ which do not intersect $F$ in any vertex. Now we have
        $$|N(N_F(S'))\cap W_1|\ge w_1-Mm\ge w_1-\sqrt n/4$$
        using \autoref{easy improvement 2}, so there is a copy of $T'$ such that $V(T')\cap V(F)=\emptyset$ and $V(T')\subseteq N(N_F(S'))$. Replacing $S'$ with the corresponding vertices in $T'$ gives a copy of $F$ in $G$, a contradiction.
    \end{proof}
    \begin{claim} \label{structure of w1 claim 5}
        If the tree components of $G[W_1]$ of order at most $k-1$ are $T_1,\ldots,T_j$, then we have
        $$\sum_{i=1}^j|V(T_i)|\le (k-1)^2k.$$
    \end{claim}
    \begin{proof}
        If not, then for some $\ell\le k-1$, there are $k$ tree components $T_1,\ldots,T_k$ of order $\ell$. By \autoref{structure of w1 claim 4}, if $G'$ is obtained from $G$ by replacing $T_1\cup\cdots\cup T_k$ with $\ell\cdot T'$, then $G'$ is $\mathcal F$-free. Also, $e(G')>e(G)$ which contradicts $G\in\mathrm{EX}(n,\mathcal F)$.
    \end{proof}
    Combining \autoref{structure of w1 claim 1}, \autoref{structure of w1 claim 2}, and \autoref{structure of w1 claim 5} gives the result (we put any trees of order $k$ which occur at most $M(m+1)$ times in $W_1^m$).
\end{proof}

\begin{lemma} \label{structure of w1 2}
    Every vertex in $W_1^o$ is adjacent to all vertices in $V-W_1$.
\end{lemma}
\begin{proof}
    Assume for a contradiction that $v\in W_1^o$ and  $N(v)\not\supseteq V-W_1$, so there is some $u\not\in W_1$ with $v\not\sim u$. Then the graph $G'$ obtained by adding the edge $uv$ contains some $F\in\mathcal F$. Let the component of $G[W_1]$ containing $v$ be $T$. By \autoref{structure of w1}, there are more than $M(m+1)$ disjoint copies of $T$ in $G[W_1]$, and more than $Mm$ of these do not intersect $F$ in any vertex. Then from \autoref{easy improvement 2} we have
    $$|N(N_F(v))\cap W_1|\ge w_1-Mm$$
    so some copy of $T$ is contained in $N(N_F(v))$. This implies $F\subseteq G$, a contradiction.
\end{proof}

At this point we know that $G$ can be obtained from $K_{W_1,\ldots,W_r}$ by embedding disjoint trees of order $k$ in $W_1$, then deleting and adding $O(1)$ edges. 

\begin{lemma} \label{OVXT edge}
    Suppose $H$ is any graph obtained from $K_{W_1,\ldots,W_r}$ by embedding disjoint trees in $W_1$, then deleting and adding $O(1)$ edges. Let $W_i^o,W_i^m$ be defined as above, and assume that $|W_i^o|\to\infty$, $|W_i^m|=O(1)$ for each $i$. For any $i,j$, let $H_{ij}$ be defined by deleting a vertex $u$ in $W_i^o$ such that $d_{W_i}(u)\le 1$ and adding a vertex $v$ whose neighborhood is $V-W_j-u$. If $H$ is $\mathcal F$-free then so is $H_{ij}$. Moreover, $e(H_{ij})-e(G)=|W_i|-|W_j|-1$ if $i\ne 1$ and $e(H_{ij})-e(G)\ge |W_i|-|W_j|-2$ if $i=1$.
\end{lemma}
\begin{proof}
    Suppose $H_{ij}$ has a copy of some $F\in\mathcal F$. If $H$ is $\mathcal F$-free, we must have that $v\in V(F)$. Since $|W_j|\to\infty$ and $|V(F)|\le M$, there is some $v'\in W_j^o-V(F)$, and $N_H(v')\supseteq N_{H_{ij}}(v)$. Thus, replacing $v$ by $v'$ gives a copy of $F$ in $H$, a contradiction. The second claim is obvious.
\end{proof}

Below we refer to $H_{ij}$ for various graphs $H$ without restating the full setup when it is clear from context. We will make statements such as ``$H$ satisfies $|w_i-w_j|\le a$ for all $i,j$," where the $w_i$ are assumed to be the values of $|W_i|$ which satisfy the assumptions of \autoref{OVXT edge}.

\begin{lemma} \label{balance 1}
    For any $i,j\in\{2,\ldots,r\}$, we have $|w_i-w_j|\le 1$.
\end{lemma}
\begin{proof}
    If not, then there is some $i,j\ge 2$ such that $w_i\ge w_j+2$. By \autoref{OVXT edge}, $G_{ij}$ is $F$-free and $e(G_{ij})-e(G)=w_i-w_j-1\ge 1$, which is a contradiction.
\end{proof}

\begin{lemma} \label{balance 2}
    For any $i\in\{2,\ldots,r\}$, we have $w_1-2\le w_i\le w_1+1$.
\end{lemma}
\begin{proof}
    If $w_1-3\ge w_i$ then by \autoref{OVXT edge} $G_{1i}$ is $\mathcal F$-free and $e(G_{1i})-e(G)\ge w_1-w_i-2\ge 1$, a contradiction. If $w_i\ge w_1+2$ then we have that $G_{i1}$ is $\mathcal F$-free and $e(G_{i1})-e(G)= w_i-w_1-1\ge 1$, a contradiction.
\end{proof}

\begin{proof}[Proof of \autoref{structure lemma 1}] By the above, we are done unless $w_1=w_i+2$ for some $i$. Then by \autoref{OVXT edge}, $G_{1i}$ is $\mathcal F$-free and $e(G_{1i})-e(G)\ge 0$, i.e. $G_{1i}\in\mathrm{EX}(n,\mathcal F)$. One can check that $G_{1i}$ satisfies $|w_i-w_j|\le 1$ for all $i,j$.
\end{proof}

\subsection{Structure of spectral extremal graphs} \label{section strucutre of spectral extremal}

We now prove \autoref{main theorem}. Let $\mathcal F$ be as in \autoref{main theorem} and let $k=k(\mathcal F)$ as in \autoref{k value}. If $k=1$, then the conclusion of \autoref{main theorem} follows from \autoref{non bipartite paper result}, so we will assume that $k\ge 2$. Similarly to the previous subsection, we have the following lemma.

\begin{lemma} \label{bad graphs spectral}
    For some large enough even constant $m$, the following graphs are not $\mathcal F$-free.
    \begin{itemize}
        \item $K_{1,m}+T_{m(r-1),r-1}$
        \item $M_m+M_m+T_{m(r-2),r-2}$
        \item $(T_1\cup\cdots\cup T_m)+T_{m(r-1),r-1}$, for any trees $T_i$ of order $k+1$
        \item $(C_1\cup\cdots\cup C_m)+T_{m(r-1),r-1}$, for any cycles $C_i$ of order at most $k$
    \end{itemize}
\end{lemma}

Let $m$ be a large enough even constant satisfying \autoref{bad graphs spectral}. Let $M=\max\{|V(F)|: F\in\mathcal F\}$. Suppose $G\in\mathrm{SPEX}(n,\mathcal F)$, where $n$ is large enough. Choose positive constants $\varepsilon,\varepsilon_1,\varepsilon_2,\varepsilon_3$,$\varepsilon_4,\varepsilon_5,\varepsilon_6$
satisfying the following inequalities (we do not attempt to avoid redundancy or optimize the constants).
\begin{center}
\begin{tabular}{p{0.4\textwidth}p{0.5\textwidth}}
\begin{itemize}
    \item[(i)] $2(r+1)\sqrt\varepsilon<\min\{\varepsilon_1,\varepsilon_3\}$
    \item[(ii)] $2\varepsilon_1+\sqrt\varepsilon<1/r$
    \item[(iii)] $2(r+1)^2(\sqrt\varepsilon+\varepsilon_1)<\varepsilon_2$
    \item[(iv)] $M\sqrt\varepsilon<1/r$
    \item[(v)] $\sqrt\varepsilon+\varepsilon_2+2\varepsilon_3<1/r$
    \item[(vi)] $\sqrt\varepsilon+\varepsilon_2+\varepsilon_3<\varepsilon_4$
    \item[(vii)] $(2mr+1)\sqrt\varepsilon<\varepsilon_3$
    \item[(viii)] $\max\{\varepsilon_4+\sqrt\varepsilon,3\sqrt\varepsilon\}<\varepsilon_5$
    \item[(ix)] $\sqrt\varepsilon+m(k+1)\varepsilon_5<1/r$
\end{itemize}&
\begin{itemize}
    \item[(x)] $\sqrt\varepsilon+m(k+r)\varepsilon_5<1/r$
    \item[(xi)] $\sqrt\varepsilon+M\varepsilon_5<1/r$
    \item[(xii)] $\frac{r(\varepsilon_1+(r-1)\varepsilon_5)}{r-1}<\varepsilon_6$
    \item[(xiii)] $\frac{r}{r-1}(\varepsilon_5+\varepsilon_1+3\sqrt\varepsilon+\varepsilon_2/r+(r-1)\varepsilon_5)<\varepsilon_6$
    \item[(xiv)] $(m+1)\varepsilon_6<1$
    \item[(xv)] $\sqrt\varepsilon+2mr\varepsilon_5<1/r$
    \item[(xvi)] $2\varepsilon_6<1-\frac{k(k-2)}{(k-1)^2}$
    \item[(xvii)] $2\varepsilon_6<1-\frac{k(k-2)}{(k-1)^2}$.
\end{itemize}
\end{tabular}
\end{center}
This is possible, as seen by choosing them in the order $\varepsilon_6,\varepsilon_5,\varepsilon_4,\varepsilon_3,\varepsilon_2,\varepsilon_1,\varepsilon$.

We will use the spectral version of the stability theorem.

\begin{lemma}[Desai-Kang-Li-Ni-Tait-Wang \cite{desai2022spectral}] \label{spectral stability}
    If $n$ is large enough then $G$ can be obtained from $T_{n,r}$ by adding and deleting at most $\varepsilon n^2$ edges.
\end{lemma}

Therefore, let $V(G)=V_1\sqcup\cdots\sqcup V_r$, where $|V_i|\in\{\lfloor n/r\rfloor,\lceil n/r\rceil\}$ and $G$ is obtained from $K_{V_1,\ldots,V_r}$ by adding a graph $A$ and removing a graph $B$ with $e(A)+e(B)\le\varepsilon n^2$. Let $X=\{v\in V:d_B(x)\ge 2\sqrt\varepsilon n\}$ and note that we immediately have $|X|\le\sqrt{\varepsilon}n$.

The proof of the following lemma follows the lines of Lemma 3.1 of \cite{wang2023conjecture}.

\begin{lemma} \label{connected}
    The graph $G$ is connected.
\end{lemma}
\begin{proof}
    Suppose $G$ is not connected. Then $\lambda(G)=\lambda(G_1)$, where $G_1$ is a largest component of $G$. Let $G'$ be obtained from $G_1$ by choosing any vertex $v\in V(G_1)$ and adding a vertex $u$ pendant to $v$, along with $n-|V(G_1)|-1$ isolated vertices. Then $\lambda(G_1)>\lambda(G)\ge\lambda(T_{n,r})\to\infty$ and $G_1$ contains a copy of some $F\in\mathcal F$. On the other hand, we claim that $d_{G_1}(v)\le M$; otherwise we could replace $u$ with some other vertex to show that $G\supseteq F$. Since $v$ was arbitrary, we find that $\Delta(G_1)\le M$ and consequently $\lambda(G_1)\le M$, a contradiction.
\end{proof}

Let $x$ be the principal eigenvector of $G$, normalized so that its maximum entry is $1$. Let $1=x_z=\max\{x_v:v\in V\}$. By \autoref{connected}, the Perron-Frobenius theorem implies that all entries of $x$ are positive.

\begin{lemma} \label{z belongs somewhere}
    There exists a unique $i_1$ such that $d_{V_{i_1}}(z)\le\varepsilon_1n$ and for any $j\ne i_1$, we have $d_{V_j}(z)\ge (1/r-\varepsilon_1)n+O(1)$.
\end{lemma}
\begin{proof}
    First, note that $d(z)\ge\lambda x_z=\lambda\ge\left(1-\frac{1}{r}\right)n+O(1)$. We claim there is a unique $i\in[r]$ such that $d_{V_i-X}(z)\le(\varepsilon_1-\sqrt\varepsilon)n$ and for any $j\ne i$, $d_{V_j}(z)\ge\frac{1}{r}n-\varepsilon_1n$. Suppose there is no $i$ such that $d_{V_i-X}(z)\le(\varepsilon_1-\sqrt\varepsilon)n$. By (i), there are some $m$ neighbors of $z$ in $V_1-X$, say $v_{11},\ldots,v_{1m}$. Having found $v_{11},\ldots,v_{im}$, we have
    $$|N(\{v,v_{11},\ldots,v_{im}\})\cap (V_{i+1}-X))|\ge(\varepsilon_1-\sqrt\varepsilon-2im\sqrt\varepsilon-\sqrt\varepsilon)n\ge m$$
    since $2(r+1)\sqrt\varepsilon<\varepsilon_1$, (i). Thus, we eventually find a copy of $K_{1,m}+T_{m(r-1),r-1}$ in $G$, contradicting \autoref{bad graphs spectral}. So some such $i_1$ exists. To prove the uniqueness, suppose there are $i\ne j$ such that $d_{V_i-X}(z),d_{V_j-X}(z)\le\varepsilon_1 n$; then we have $d(z)\le\left(1-\frac{2}{r}+2\varepsilon_1\right)n+O(1)$ which is a contradiction by (ii). Hence, there is a unique $i$ such that $d_{V_i-X}(z)\le\varepsilon_1n$. Finally, since $d(z)\ge\left(1-\frac{1}{r}\right)n+O(1)$, we have that for any $j\ne i$, we have $d_{V_j}(z)\ge(1/r-\varepsilon_1)n+O(1)$.
\end{proof}

Note that $\lambda\ge\lambda(T_{n,r})=\frac{r-1}{r}n+O(1)$.

\begin{lemma} \label{eigenweight estimate 1}
    For any $v\in V-X$, we have $x_v\ge 1-\varepsilon_2$.
\end{lemma}
\begin{proof}
    First suppose $v\in V_{i_1}$. Since $v\not\in X$, we have
    $$|N(z)-N(v)|\le\varepsilon_1n+2(r-1)\sqrt\varepsilon n.$$
    and so
    $$\begin{aligned}
        \lambda x_v&\ge\lambda x_z-\sum_{u\in N(z)-N(v)}x_u\\
        &\ge\lambda x_z-(\varepsilon_1+2(r-1)\sqrt\varepsilon)n\\
        x_v&\ge 1-(\varepsilon_1+2(r-1)\sqrt\varepsilon)\frac{n}{\lambda}\\
        &\ge 1-\frac{2r-1}{r-1}\varepsilon_1\\
        &\ge 1-\frac{2r^2\varepsilon_1}{r-1}+o(1)\ge 1-2r^2\varepsilon_1.
    \end{aligned}$$
    Now if $v\not\in V_{i_1}$, then we have $$|N(z)-N(v)|\le(1/r+2r\sqrt\varepsilon)n+O(1)$$ and $$|(N(v)-N(z))\cap V_{i_1}|\ge(1/r-2\sqrt\varepsilon-\varepsilon_1)+O(1).$$
    Thus, applying the eigenweight bound above for vertices in $V_{i_1}$, we have
    $$\begin{aligned}
        \lambda x_v&\ge\lambda x_z-\sum_{u\in N(z)-N(v)}x_u+\sum_{u\in N(v)-N(z)}x_u\\
        &\ge \lambda-\frac{n}{r}-r2\sqrt\varepsilon n+\left(\frac{n}{r}-(2\sqrt{\varepsilon}+\varepsilon_1)n\right)(1-2r^2\varepsilon_1)+O(1)\\
        &\ge\lambda-\left(2(r+1)\sqrt\varepsilon+\varepsilon_1+2r\varepsilon_1\right)n+O(1)\\
        &\ge\lambda-2(r+1)(\sqrt\varepsilon+\varepsilon_1)n+O(1)\\
        x_v&\ge 1-2(r+1)(\sqrt\varepsilon+\varepsilon_1)\frac{n}{\lambda}+o(1)\ge1- \frac{2r(r+1)(\sqrt\varepsilon+\varepsilon_1)}{r-1}+o(1)\\
        &\ge 1-2(r+1)^2(\sqrt\varepsilon+\varepsilon_1).
    \end{aligned}$$
    The claim now follows from $2(r+1)^2(\sqrt\varepsilon+\varepsilon_1)<\varepsilon_2$, (iii).
\end{proof}

The next section of the proof follows the lines of the argument in \autoref{section structure of edge extremal}. Define a $V$-\textit{ordinary vertex} $v\in V_i$ to be one such that $N(v)=\bigcup_{j\ne i} V_j-X$.

\begin{lemma} \label{spectral ordinary vertex replacement}
    Let $v\in X$. Then there exists some $i$ such that $d_{V_i}(v)\le\varepsilon_3 n$ and $d_{V_j}(v)\ge n/r-\varepsilon_4 n$ for every $j\ne i$.
\end{lemma}
\begin{proof}
    Suppose not. Note that a $V$-ordinary vertex satisfies $d(u)\ge(1-1/r-\sqrt\varepsilon)n+O(1)$.
    \begin{claim} \label{spectral ordinary vertex replacement claim 1}
        If $G'$ is obtained from $G$ by deleting some vertex and adding a $V$-ordinary vertex $u\in V_i$, then $G'$ is $\mathcal F$-free.
    \end{claim}
    \begin{proof}
    We copy the proof of \autoref{ordinary vertex replacement claim 1}, using $M\sqrt\varepsilon<1/r$, (iv).
    \end{proof}
    If there are two different indices $i,i'$ such that $d_{V_i}(v),d_{V_{i'}}(v)\le\varepsilon_3n$, then we have
    $$d(v)\le(1-2/r+2\varepsilon_3)n+O(1).$$
    By replacing $v$ with a $V$-ordinary vertex, we obtain that
    $$\begin{aligned}\frac{1}{2x_v}(x^TA(G')x-x^TA(G)x)&\ge (1-1/r-\sqrt\varepsilon)(1-\varepsilon_2)n-(1-2/r+2\varepsilon_3)n+O(1)\\
    &\ge(1/r-\sqrt\varepsilon-\varepsilon_2-2\varepsilon_3)n+O(1)>0
    \end{aligned}$$
    since $\sqrt\varepsilon+\varepsilon_2+2\varepsilon_3<1/r$, (v), which by \autoref{spectral ordinary vertex replacement claim 1} contradicts $G\in\mathrm{SPEX}(n,\mathcal F)$. If there is a unique $i$ such that $d_{V_i}(v)\le\varepsilon_3n$, then by assumption there is some $j$ for which $d_{V_j}(v)\le n/r-\varepsilon_4 n$. Thus
    $$d(v)\le(1-1/r+\varepsilon_3-\varepsilon_4)n+O(1).$$
    Then replacing $v$ by a $V$-ordinary vertex gives
    $$\begin{aligned}
        \frac{1}{2x_v}(x^TA(G')x-x^TA(G)x)&\ge(1-1/r-\sqrt\varepsilon)(1-\varepsilon_2)n-(1-1/r+\varepsilon_3-\varepsilon_4)n+O(1)\\
        &\ge (-\sqrt\varepsilon-\varepsilon_2-\varepsilon_3+\varepsilon_4)n+O(1)>0
    \end{aligned}$$
    since $\sqrt\varepsilon+\varepsilon_2+\varepsilon_3<\varepsilon_4$, (vi), which by \autoref{spectral ordinary vertex replacement claim 1} contradicts $G\in\mathrm{SPEX}(n,\mathcal F)$. Finally, suppose there is no index $i$ for which $d_{V_i}(v)\le\varepsilon_3 n$. Then using $(2mr+1)\sqrt\varepsilon<\varepsilon_3$, (vii) and a similar argument to \autoref{ordinary vertex replacement}, we find a $K_{1,m}+T_{m(r-1),r-1}$, which contradicts \autoref{bad graphs spectral}.
\end{proof}

By \autoref{spectral ordinary vertex replacement}, we may define a new partition $V=W_1\sqcup\cdots\sqcup W_r$ as follows. For $v\in V_i-X$, let $v\in W_i$. If $v\in X$, then let $v\in W_i$ where $i$ is chosen so that $d_{V_i}(v)\le\varepsilon_3n$ and $d_{V_j}(v)\ge n/r-\varepsilon_4n$ for each $j\ne i$. Let $w_i:=|W_i|$. Note that $w_i\in(n/r-\sqrt\varepsilon n+O(1),n/r+\sqrt\varepsilon n+O(1))$. Thus for each $v\in W_i$, for each $j\ne i$, either $d_{W_j}(v)\ge w_j-(\varepsilon_4+\sqrt\varepsilon)n+O(1)\ge w_j-\varepsilon_5n$ or $d_{W_j}(v)\ge w_j-(2\sqrt\varepsilon+\sqrt\varepsilon)n+O(1)\ge w_j-\varepsilon_5n$, since $\max\{\varepsilon_4+\sqrt\varepsilon,3\sqrt\varepsilon\}<\varepsilon_5$, (viii). Also, call a vertex $v\in W_i$ $W$-\textit{ordinary} if $N(v)=V-W_i$.

\begin{lemma} \label{automatic joining spectral}
    Let $K$ be a fixed graph of order at most $m(k+1)$. For $n$ large enough, if $G[W_i]\supseteq K$ for some $i$ then $G\supseteq K+T_{m(r-1),r-1}$.
\end{lemma}
\begin{proof}
    We copy the proof of \autoref{automatic joining}, using the inequalities $\sqrt\varepsilon+m(k+1)\varepsilon_5<1/r$, (ix), and $\sqrt\varepsilon+m(k+r)\varepsilon_5<1/r$, (x).
\end{proof}

\begin{lemma} \label{ordinary vertex replacement 2 spectral}
    Let $G'$ be obtained from $G$ by deleting some vertex and adding a $W$-ordinary vertex $u\in V_i$. Then $G'$ is $\mathcal F$-free.
\end{lemma}
\begin{proof}
    We copy the proof of \autoref{ordinary vertex replacement 2}, using the inequality $\sqrt\varepsilon+M\varepsilon_5<1/r$, (xi).
\end{proof}

\begin{lemma} \label{spectral improvement}
    For each $W_i$, for each $v\in W_i$, we have $d_{W_i}(v)\le m$.
\end{lemma}
\begin{proof}
    This follows immediately from \autoref{automatic joining spectral} and \autoref{bad graphs spectral}.
\end{proof}

\begin{lemma} \label{eigenweight estimate 2}
    For every $v\in V$, we have $x_v\ge 1-\varepsilon_6$.
\end{lemma}
\begin{proof}
    Let $z\in W_{i_1}$. If $v\in W_{i_1}$, then we have
    $$|N(z)-N(v)|\le(\varepsilon_1+(r-1)\varepsilon_5)n$$
    so
    $$\begin{aligned}
        \lambda x_v&\ge\lambda x_z-\sum_{u\in N(z)-N(v)}x_u\\
        &\ge\lambda-(\varepsilon_1+(r-1)\varepsilon_5)n\\
        x_v&\ge 1-\frac{(\varepsilon_1+(r-1)\varepsilon_5)n}{\lambda}\\
        &\ge 1-\frac{r(\varepsilon_1+(r-1)\varepsilon_5)}{r-1}+o(1).\\
    \end{aligned}$$
    If $v\in W_j\ne W_{i_1}$, then we have
    $$|N(z)-N(v)|\le w_{j}+(r-1)\varepsilon_5n\le(1/r+\sqrt\varepsilon+(r-1)\varepsilon_5)n+O(1)$$
    and
    $$|(N(v)-N(z))-X|\ge w_{i_1}-\varepsilon_5n-(\varepsilon_1+\sqrt\varepsilon)n-\sqrt\varepsilon n+O(1)\\
    \ge(1/r-\varepsilon_5-\varepsilon_1-2\sqrt\varepsilon)n+O(1).$$
    Thus using \autoref{eigenweight estimate 1} we have
    $$\begin{aligned}
        \lambda x_v&=\lambda x_z+\sum_{u\in N(v)-N(z)}x_u-\sum_{u\in N(z)-N(v)}x_u\\
        &\ge \lambda+(1/r-\varepsilon_5-\varepsilon_1-2\sqrt\varepsilon)(1-\varepsilon_2)n-(1/r+\sqrt\varepsilon+(r-1)\varepsilon_5)n+O(1)\\
        &\ge\lambda-(\varepsilon_5+\varepsilon_1+3\sqrt\varepsilon+\varepsilon_2/r+(r-1)\varepsilon_5)n+O(1)\\
        x_v&\ge 1-\frac{r}{r-1}(\varepsilon_5+\varepsilon_1+3\sqrt\varepsilon+\varepsilon_2/r+(r-1)\varepsilon_5)n+o(1)
    \end{aligned}$$
    and the conclusion follows from $\frac{r(\varepsilon_1+(r-1)\varepsilon_5)}{r-1}<\varepsilon_6$, (xii) and $\frac{r}{r-1}(\varepsilon_5+\varepsilon_1+3\sqrt\varepsilon+\varepsilon_2/r+(r-1)\varepsilon_5)<\varepsilon_6$, (xiii).
\end{proof}

\begin{lemma} \label{spectral improvement 2}
    For each $v\in W_i$, for each $j\ne i$, we have $d_{W_j}(v)\ge w_j-m$.
\end{lemma}
\begin{proof}
    Suppose not. By \autoref{ordinary vertex replacement 2 spectral}, the graph $G'$ obtained by replacing $v$ with a $W$-ordinary vertex is $\mathcal F$-free, and by \autoref{spectral improvement} and \autoref{eigenweight estimate 2} we have
    $$\begin{aligned}
        \frac{1}{2x_v}(x^TA(G')x-x^TA(G)x)&\ge(m+1)(1-\varepsilon_6)-m=1-(m+1)\varepsilon_6>0
    \end{aligned}$$
    since $(m+1)\varepsilon_6<1$, (xiv), a contradiction.
\end{proof}

\begin{lemma} \label{spectral one part}
    For some constant $C$, there is at most one $i$ such that $e(W_i)\ge C$.
\end{lemma}
\begin{proof}
    We copy the proof of \autoref{strong one part}, using $\sqrt\varepsilon+2mr\varepsilon_5<1/r$ (xv).
\end{proof}

Assume without loss of generality that $W_1$ is the part for which $e(W_i)\ge C$ is possible, so that $e(W_i)=O(1)$ for all $i\ne 1$. Observe that for each $i\ne 1$, we may write $W_i=W_i^o\sqcup W_i^m$, where $e(W_i^o)=0$ and $|W_i^m|=O(1)$. We will also define $W^o=\bigcup_{i=1}^rW_i^o$ and $W^m=\bigcup_{i=1}^rW_i^m$.

\begin{lemma} \label{spectral ordinary vertices}
    For each $i\ne 1$, for each $v\in W_i^o$, $N(v)=V-W_i$.
\end{lemma}
\begin{proof}
    We copy the proof of \autoref{ordinary vertices}, using the fact that adding an edge to a connected graph increases its spectral radius.
\end{proof}

\begin{lemma} \label{spectral general lemma}
    Suppose $G'$ can be obtained from $G$ by deleting $a$ edges and adding $b$ edges, where $a/b<1-2\varepsilon_6$. Then $\lambda(G')>\lambda(G)$.
\end{lemma}
\begin{proof}
    By \autoref{eigenweight estimate 2}, we have
    $$\frac{1}{2}(x^TA(G')x-x^TA(G)x)\ge b(1-\varepsilon_6)^2-a>b(1-2\varepsilon_6)-a>0.$$
\end{proof}

\begin{lemma} \label{spectral structure of w1}
    We can write $W_1=W_1^o\sqcup W_1^m$, where $W_1^o$ is a union of trees of fixed order $k$, each tree component of $G[W_1^o]$ occurs more than $M(m+1)$ times, and $|W_1^m|=O(1)$.
\end{lemma}
\begin{proof}
    We prove the result in a series of claims.
    \begin{claim} \label{structure of w1 spectral claim 1}
        The number of components of $G[W_1]$ whose order is at least $k+1$ is at most $m$.
    \end{claim}
    \begin{proof}
        We copy the proof of \autoref{structure of w1 claim 1}.
    \end{proof}
    \begin{claim} \label{structure of w1 spectral claim 2}
        The graph $G[W_1]$ has at most $m$ disjoint cycles of order at most $k$.
    \end{claim}
    \begin{proof}
        We copy the proof of \autoref{structure of w1 claim 2}.
    \end{proof}
    \begin{claim} \label{structure of w1 spectral claim 3}
        The graph $G[W_1]$ has at least $\sqrt n$ tree components of order $k$.
    \end{claim}
    \begin{proof}
        Similarly to \autoref{structure of w1 claim 3}, we find an $\mathcal F$-free graph $G'$ obtained from $G$ by adding $b=w_1\frac{k-1}{k}+o(n)$ edges and deleting $a=w_1\frac{k-2}{k-1}+o(n)$ edges. Then
        $$\frac{a}{b}\le\frac{k(k-2)}{(k-1)^2}+o(1)<1-2\varepsilon_6$$
        since $2\varepsilon_6<1-\frac{k(k-2)}{(k-1)^2}$, (xvi),
        so \autoref{spectral general lemma} gives $\lambda(G')>\lambda(G)$, a contradiction.
    \end{proof}
    Let $T'$ be a tree occurring at least $\sqrt n$ times as guaranteed by \autoref{structure of w1 spectral claim 3}.
    \begin{claim} \label{strucure of w1 spectral claim 4}
        For any set $S\subseteq W_1$ with $|S|=k$, the graph $G'$ obtained by deleting $E(S)$ and adding a copy of $T'$ on $S$ is $\mathcal F$-free.
    \end{claim}
    \begin{proof}
        We copy the proof of \autoref{structure of w1 claim 4}.
    \end{proof}
    \begin{claim} \label{structure of w1 spectral claim 5}
        If the tree components of $G[W_1]$ of order at most $k-1$ are $T_1,\ldots,T_a$, then we have
        $$\sum_{i=1}^a|V(T_i)|\le(k-1)^2k.$$
    \end{claim}
    \begin{proof}
        As in \autoref{structure of w1 claim 5}, we obtain an $\mathcal F$-free graph $G'$ obtained from $G$ by deleting $a=k(\ell-1)$ edges and adding $b=(k-1)\ell$ edges, where $\ell\le k-1$. Then we have
        $$\frac{a}{b}\le\frac{k(\ell-1)}{(k-1)\ell}\le\frac{k(k-2)}{(k-1)^2}<1-2\varepsilon_6$$
        since $2\varepsilon_6<1-\frac{k(k-2)}{(k-1)^2}$, (xvii). Hence, \autoref{spectral general lemma} gives $\lambda(G')>\lambda(G)$, contradicting $G\in\mathrm{SPEX}(n,\mathcal F)$.
    \end{proof}
    Combining \autoref{structure of w1 spectral claim 1}, \autoref{structure of w1 spectral claim 2}, and \autoref{structure of w1 spectral claim 5}, we obtain the result (we put any trees of order $k$ which occur at most $M(m+1)$ times in $W_1^m$).
\end{proof}

\begin{lemma} \label{spectral structure of w1 2}
    For each $v\in W_1^o$, $N(v)\supseteq V-W_i$.
\end{lemma}
\begin{proof}
    We copy the proof of \autoref{structure of w1 2}, using the fact that adding an edge to a connected graph increases the spectral radius.
\end{proof}

Now we know that $G$ can be obtained from $K_{W_1,\ldots,W_r}$ by embedding trees of order $k$ in $W_1$, then adding and deleting $O(1)$ edges. Now we determine the values of the $w_i$. Since all vertices in $W_i^o$ ($i\ne 1$) have the same neighborhood, they have the same eigenweight, so let $x_i$ be the common eigenweight of vertices in $W_i^o$. Let $x_1=\frac{1}{\lambda}\sum_{v\not\in W_1}x_v$. Observe that $\lambda=\Theta(n)$, a fact we will use frequently below.

\begin{lemma} \label{eigenweight estimate 3}
    For each $i\ne 1$, $v\in W_i$, we have $x_v=x_i+O(n^{-1})$.
\end{lemma}
\begin{proof}
    Let $u\in W_i^o$, and observe that $N(v)=N(u)\cup A'-B'$, where $a:=|A'|=O(1)$ and $b:=|B'|=O(1)$. Thus,
    $$\begin{aligned}
        \lambda x_v&=\lambda x_u+\sum_{w\in A'}x_w-\sum_{w\in B'}x_b\\
        |x_v-x_i|&=\frac{1}{\lambda}\left|\sum_{w\in A'}x_w+\sum_{w\in B}x_w\right|
        \le\frac{1}{\lambda}(a+b)=O(n^{-1}).
    \end{aligned}$$
\end{proof}

\begin{lemma} \label{eigenweight estimate 3.1}
    For any $1<i<j$, we have
    $$x_j=\frac{\lambda+w_i}{\lambda+w_j}x_i+O(n^{-2}).$$
\end{lemma}
\begin{proof}
    We have
    $$\begin{aligned}
        \lambda x_j&=\lambda x_i+x_i(w_i-|W_i^m|)+\sum_{w\in W_i^m}x_w-x_j(w_j-|W_j^m|)-\sum_{w\in W_j^m}x_w\\
        &=\lambda x_i+w_ix_i+|W_i^m|O(n^{-1})-w_jx_j-|W_j^m|O(n^{-1})\\
        &=\lambda x_i+w_ix_i-w_jx_j+O(n^{-1}).
    \end{aligned}$$
    Solving for $x_j$ gives the result.
\end{proof}

\begin{lemma} \label{eigenweight estimate 4}
    For $v\in W_1$, we have $x_v=x_1+O(n^{-1})$.
\end{lemma}
\begin{proof}
    We follow the proof of \autoref{eigenweight estimate 3}.
\end{proof}

\begin{lemma} \label{eigenweight estimate 5}
    For $v\in W_1$, we have $x_v=x_1+\frac{d_{W_1}(v)}{\lambda}x_1+O(n^{-2})$.
\end{lemma}
\begin{proof}
    We have
    $$\begin{aligned}
        \lambda x_v&=\lambda x_1+\sum_{u\in N_{W_1}(v)}x_u\\
        &=\lambda x_1+d_{W_1}(v)(x_1+O(n^{-1}))
    \end{aligned}$$
    and since $\Delta(G[W_1])$ is bounded, solving for $x_v$ gives the result.
\end{proof}

The spectral version of \autoref{OVXT edge} gives the following fact: for any $u\in W_i^o$, for any $j\ne i$, we have $\lambda(G_{ij})\le\lambda(G)$ and so with $u$ being the transferred vertex, we have
\begin{equation} \label{equation 3}
0\ge\frac{1}{2x_u}(x^TA(G_{ij})x-x^TA(G)x)=\sum_{v\in W_i-\{u\}-(N(u)\cap W_i)}x_v-\sum_{v\in W_j}x_v
\end{equation}

For each $i$ write $X_i=\sum_{u\in W_i}x_u$. Note that for $i\ne 1$, $X_i=w_ix_i+O(n^{-1})$ and when $i=1$ we have
$$X_1=\sum_{u\in W_1}(x_1+\frac{d_{W_1}(u)}{\lambda}x_1+O(n^{-2}))=w_1x_1+\frac{2e(W_1)}{\lambda}x_1+O(n^{-1})=w_1x_1+O(1).$$

\begin{lemma} \label{eigenweight estimate 5.1}
    For any $i>1$, we have $x_i=\frac{\lambda+w_1+2e(W_1)/\lambda}{\lambda+w_i}x_1+O(n^{-2})$.
\end{lemma}
\begin{proof}
    We have
    $$\begin{aligned}\lambda x_i&=\lambda x_1-w_ix_i+X_1\\
    &=\lambda x_1-w_ix_i+w_1x_1+\frac{2e(W_1)}{\lambda}x_1+O(n^{-1})\end{aligned}$$
    and solving for $x_i$ gives the result.
\end{proof}

\begin{lemma} \label{initial estimates}
    We have $w_i=n/r+O(1)$ for each $i$.
\end{lemma}
\begin{proof}
    It suffices to show that for each $i>1$, we have $|
    w_1-w_i|=O(1)$. Taking \autoref{equation 3} with respect to $G_{1i}$ gives
    $$\begin{aligned}
        0&\ge w_1x_1-w_ix_i+O(1)\\
        &=w_1x_1-w_i\frac{\lambda+w_1+O(1)}{\lambda+w_i}x_1+O(1)=\left(w_1-w_i\frac{\lambda+w_i}{\lambda+w_1}\right)x_1\\
        O(n)&\ge(\lambda+w_i)w_1-(\lambda+w_1)w_i=\lambda(w_1-w_i)
    \end{aligned}$$
    from which we obtain $w_1-w_i\le O(1)$. Similarly, taking \autoref{equation 3} with respect to $G_{i1}$ gives
    $$\begin{aligned}
        0&\ge w_ix_i-w_1x_1+O(1)\\
        &=w_i\frac{\lambda+w_1+O(1)}{\lambda+w_i}x_1-w_1x_1+O(1)\\
        &=\left(w_i\frac{\lambda+w_1}{\lambda+w_i}-w_1\right)x_1+O(1)\\
        O(n)&\ge(\lambda+w_1)w_i-(\lambda+w_i)w_1=\lambda(w_i-w_1)
    \end{aligned}$$
    from which we obtain $w_i-w_1\le O(1)$.
\end{proof}

Combining \autoref{eigenweight estimate 5.1} and \autoref{initial estimates} now gives $x_i=x_1+O(n^{-1})$ for every $i$. Furthermore, combining \autoref{initial estimates} with the bound $\lambda\le\Delta(G)$ proves $\lambda\le\frac{r-1}{r}n+O(1)$, hence $\lambda=\frac{r-1}{r}n+O(1)$. These estimates will be useful in the rest of the argument.

\begin{lemma} \label{spectral balance 1} We have $|w_i-w_j|\le 1$ for each $i,j>1$.
\end{lemma}
\begin{proof}
    If there is some $i,j>1$ such that $w_j\ge w_i+2$, then taking \autoref{equation 3} with respect to $G_{ji}$ gives
    $$\begin{aligned}
        0\ge&\sum_{v\in W_j-u}x_v-\sum_{v\in W_i}x_v\\
        =&(w_j-1-|W_j^m|)\left(\frac{\lambda+w_i}{\lambda+w_j}x_i+O(n^{-2})\right)+|W_j^m|\left(\frac{\lambda+w_i}{\lambda+w_j}x_i+O(n^{-1})\right)\\
        &-(w_i-|W_i^m|)x_i-|W_i^m|(x_i+O(n^{-1}))\\
        =&\left((w_j-1)\frac{\lambda+w_i}{\lambda+w_j}-w_i\right)x_i+O(n^{-1}).
    \end{aligned}$$
    Now since $w_i<w_j$, we have $w_i<n/2$, and since $r\ge 3$ we have $\lambda\ge\left(1-\frac{1}{r}\right)n+O(1)\ge 2n/3+O(1)$. This implies that $\lambda>w_i+n^{1/2}$, which together with $w_j\ge w_i+2$ implies $\lambda(w_j-w_i-1)>w_i+n^{1/2}$, so $(w_j-1)\frac{\lambda+w_i}{\lambda+w_j}-w_i\ge\Omega(n^{-1/2})$ and so the expression displayed above is positive, a contradiction.
\end{proof}

\begin{lemma} \label{spectral balance 2}
    For all $i>1$, we have $w_1\le w_i+2$.
\end{lemma}
\begin{proof}
    Otherwise, we have $w_1\ge w_i+3$.
    Then taking \autoref{equation 3} with respect to $G_{1i}$, where $u$ is pendant in $G[W_1]$, gives
    $$\begin{aligned}0&\ge w_1x_1+\frac{2e(W_1)}{\lambda}x_1-2x_1-w_ix_i+O(n^{-1})\\
    &=w_1x_1-2+\frac{2e(W_1)}{\lambda}x_1-w_i\frac{\lambda+w_1+2e(W_1)/\lambda}{\lambda+w_i}x_1+O(n^{-1})\\
    &=(w_1-w_i)\left(1-\frac{w_i}{\lambda+w_i}\right)x_1-2+2e(W_1)\left(\frac{1}{\lambda}-\frac{w_i}{\lambda(\lambda+w_i)}\right)x_1+O(n^{-1})\\
    &\ge 3\left(1-\frac{w_i}{\lambda+w_i}\right)-2+\frac{2e(W_1)}{\lambda}\left(1-\frac{w_i}{\lambda+w_i}\right)+O(n^{-1})\\
    &=3\left(1-\frac{1}{r}\right)-2+\frac{2e(W_1)}{\lambda}\left(1-\frac{w_i}{\lambda+w_i}\right)+O(n^{-1})>0.
    \end{aligned}$$
    Now if $e(W_1)=o(n)$ then we must have $k(\mathcal F)=1$ and \autoref{non bipartite paper result} gives $G\in\mathrm{EX}(n,\mathcal F)$. Otherwise, we have $\frac{2e(W_1)}{\lambda}\ge\frac{1}{r-1}+o(1)$ and $1-\frac{w_i}{\lambda+w_i}=1-\frac{1}{r}+o(1)$, so that
    $$0\ge 3\left(1-\frac{1}{r}\right)-2+\frac{1}{r}+o(1)\ge 1-\frac{2}{r}+o(1)>0$$
    a contradiction.
\end{proof}
    
\begin{lemma} We have $w_i\le w_1+2$ for all $i$. \label{spectral balance 2.5}
\end{lemma}
\begin{proof}
Otherwise we have $w_i\ge w_1+3$. Then taking \autoref{equation 3} with respect to $G_{i1}$ gives
    $$\begin{aligned}
        0&\ge (w_i-1)x_i-w_1x_1-\frac{2e(W_1)}{\lambda}x_1+O(n^{-1})\\
        &=w_i\frac{\lambda+w_1+2e(W_1)/\lambda}{\lambda+w_i}x_1-1-w_1x_1-\frac{2e(W_1)}{\lambda}+O(n^{-1})\\
        &=w_i\left(1+\frac{w_1-w_i}{\lambda+w_i}\right)x_1-w_1x_1-1+2e(W_1)\left(\frac{w_i}{\lambda(\lambda+w_i)}-\frac{1}{\lambda}\right)+O(n^{-1})\\
        &=(w_i-w_1)\left(1-\frac{w_i}{\lambda+w_i}\right)-1+\frac{2e(W_1)}{\lambda}\left(\frac{w_i}{\lambda+w_i}-1\right)+O(n^{-1}).\\
    \end{aligned}$$
    Now using $e(W_1)\le\frac{n}{r}+O(1)$ we obtain $\frac{2e(W_1)}{\lambda}\left(\frac{w_i}{\lambda+w_i}-1\right)\ge-\frac{2}{r}+o(1)$, so
    $$0\ge 3\left(1-\frac{1}{r}\right)-1-\frac{2}{r}+o(1)\ge 2-\frac{5}{3}+o(1)>0$$
    a contradiction.
\end{proof}

We now state a general lemma which will be useful for the rest of the proof.

\begin{lemma} \label{bounded modification lemma}
    Suppose $H$ is any graph obtained from $K_{W_1,\ldots,W_r}$ by embedding a graph of bounded degree in $W_1$, then deleting and adding $O(1)$ edges. Assume that $|W_i|=n/r+O(1)$ for each $i$. If $H'$ is obtained from $H$ by deleting $a$ edges and adding $b$ edges where $a,b=O(1)$, then $\lambda(H')-\lambda(H)=2(b-a)/n+O(n^{-2})$; in particular if $a<b$ then $\lambda(H')>\lambda(H)$.
\end{lemma}
\begin{proof}
    Let $x$ be the principal eigenvector of $H$, scaled so that its largest entry is $1$. From the estimates proved above, we already have that $x_v=1+O(n^{-1})$ for every $v$. Let $A'$ be the set of deleted edges and $B'$ be the set of added edges. Then we have
    $$\begin{aligned}
        x^TA(H')x-x^TA(H)x&=2\sum_{uv\in B'}x_ux_v-2\sum_{uv\in A'}x_ux_v\\
        &=2b(1+O(n^{-1}))^2-2a(1+O(n^{-1}))^2=2(b-a)+O(n^{-1}).
    \end{aligned}$$
    Thus
    $$\begin{aligned}
        \lambda(H')-\lambda(H)&\ge\frac{x^TA(H')x}{x^Tx}-\frac{x^TA(H)x}{x^Tx}\\
        &=\frac{2(b-a)+O(n^{-1})}{n+O(1)}=\frac{2(b-a)}{n}+O(n^{-2}).
    \end{aligned}$$
    By symmetry, we also have
    $$\lambda(H)-\lambda(H')\ge\frac{2(a-b)}{n}+O(n^{-2})$$
    which proves the result.
\end{proof}

\begin{lemma} \label{spectral balance 4}
    For each $i=2,\ldots,r$, we have $w_i\le w_1+1$.
\end{lemma}
\begin{proof}
    If $w_i=w_1+2$, then as in \autoref{spectral balance 2.5}, taking \autoref{equation 3} with respect to $G_{i1}$ gives
    $$0\ge 2\left(1-\frac{1}{r}\right)-1+\frac{2e(W_1)}{\lambda}\left(\frac{w_i}{\lambda+w_i}-1\right)+O(n^{-1}).$$
    We have $2e(W_1)\le\frac{2(k-1)}{k}\frac{n}{r}+O(1)$, so we obtain
    $$0\ge 1-\frac{2}{r}-\frac{k-1}{k}\frac{2}{r}+o(1)$$
    which gives a contradiction if $r\ge 4$. If $r=3$, we require a more detailed argument. Without loss of generality, $i=2$. Let $G'$ be obtained from $G$ by converting every vertex in $W^m$ to an $W$-ordinary vertex in the same part, let $a=e(G')-e(G)=O(1)$, and let $G_{21}'=(G')_{21}$. Observe by \autoref{bounded modification lemma} that $\lambda(G')=\lambda(G)+2a/n+O(n^{-2})$ and $\lambda(G_{21})=\lambda(G_{21}')-2a/n+O(n^{-2})$, and so it suffices to show that $\lambda(G_{21}')-\lambda(G')\ge\delta/n$ for some constant $\delta>0$. Let $V(G_{21}')$ have partition $U_1\sqcup U_2\sqcup U_3$ so that $|U_1|=w_1+1$, $|U_2|=w_2-1$, and $|U_3|=w_3$. Note that according to the previous lemmas, we either have $w_1=w_2-2=w_3-1$ or $w_1=w_2-2=w_3-2$. Let $\mu=\lambda(G')$, $\rho=\lambda(G_{21}')$, with corresponding eigenvectors $y,z$, respectively, scaled so that $y_1:=\frac{1}{\mu}\sum_{v\not\in W_1}x_v=1$ and $z_1:=\frac{1}{\rho}\sum_{v\not\in U_1}z_v=1$. Let $y_2,y_3$ be the common eigenweights of vertices in $W_2,W_3$, and let $z_2,z_3$ be the common eigenweights of vertices in $U_2,U_3$.
    \begin{claim} \label{spectral balance 4 claim 1}
        We have $\rho-\mu=O(n^{-1})$,
    \end{claim}
    \begin{proof}
        Observe that $G_{21}'$ can be obtained from $G'$ by deleting $w_1$ edges and adding $w_2$ edges, so
        $$\frac{1}{2}(\lambda(G_{21}')-\lambda(G'))\ge\frac{-w_1(1+O(n^{-1}))^2+w_2(1+O(n^{-1}))^2}{y^Ty}=\frac{O(1)}{y^Ty}=O(n^{-1}).$$
        The reverse direction is proved similarly.
    \end{proof}
    Now for $v\in W_1$ we obtain
    $$\begin{aligned}
        \mu^2y_v&=\mu\left(\mu y_1+\sum_{u\in N_{W_1}(v)}x_u\right)\\
        &=\mu^2+\mu d_{W_1}(v)+d_{W_1}'(v)+O(n^{-1})\\
        y_v&=1+\frac{d_{W_1}(v)}{\mu}+\frac{d_{W_1}'(v)}{\mu^2}+O(n^{-3}).
    \end{aligned}$$
    Similarly for $v\in W_1$,
    $$\begin{aligned}
        z_v&=1+\frac{d_{W_1}(v)}{\rho}+\frac{d_{W_1}'(v)}{\rho^2}+O(n^{-3})\\
        &=1+\frac{d_{W_1}(v)}{\mu}+\frac{d_{W_1}'(v)}{\mu^2}+O(n^{-3}).
    \end{aligned}$$
    and for the vertex $v\in U_1-W_1$ we have $z_v=1$.
    Let $Y_1=\sum_{v\in W_1}y_v$, $Z_1=\sum_{v\in U_1}z_v$. Then using the eigenweight estimates above we obtain
    $$\begin{aligned}
        Y_1&=w_1+\frac{2e(W_1)}{\mu}+\frac{\sum_{v\in W_1}d_{W_1}'(v)}{\mu^2}+O(n^{-2})\\
        Z_1&=w_1+1+\frac{2e(W_1)}{\mu}+\frac{\sum_{v\in W_1}d_{W_1}'(v)}{\mu^2}+O(n^{-2})=Y_1+1+O(n^{-2}).
    \end{aligned}$$
    In particular, $Y_1,Z_1=w_1+O(1)$. From the first-degree eigenvector equations we obtain
    $$\begin{aligned}
        y_2&=1+\frac{Y_1-w_2}{\mu+w_2}\\
        y_3&=1+\frac{Y_1-w_3}{\mu+w_3}\\
        z_2&=1+\frac{Z_1-w_2+1}{\rho+w_2-1}=1+\frac{Y_2-w_2+2}{\mu+w_2-1}+O(n^{-3})\\
        &=1+\frac{Y_2-w_2+2}{\mu+w_2}+\frac{Y_2-w_2+2}{(\mu+w_2)(\mu+w_2-1)}+O(n^{-3})\\
        z_3&=1+\frac{Z_1-w_3}{\rho+w_3}=1+\frac{Y_1-w_3+1}{\mu+w_3}+O(n^{-3}).
    \end{aligned}$$
    Then
    \begin{equation} \label{spectral balance 4 equation 1}
    \begin{aligned}
        \mu=&\,\mu y_1=w_2\left(1+\frac{Y_1-w_2}{\mu+w_2}\right)+w_3\left(1+\frac{Y_1-w_3}{\mu+w_3}\right)\\
        \rho=&\,\rho z_1=(w_2-1)\left(1+\frac{Y_1-w_2+2}{\mu+w_2}+\frac{Y_1-w_2+2}{(\mu+w_2)(\mu+w_2-1)}\right)\\
        &+w_3\left(1+\frac{Y_1-w_3+1}{\mu+w_3}\right)+O(n^{-2})
    \end{aligned}
    \end{equation}
    hence
    $$\begin{aligned}
        \rho-\mu=&-1+\frac{3w_2-Y_1-2}{\mu+w_2}+\frac{(w_2-1)(Y_1-w_2+2)}{(\mu+w_2)(\mu+w_2-1)}+\frac{w_3}{\mu+w_3}+O(n^{-2})\\
    \end{aligned}$$
    Now $Y_1-w_2=\frac{2e(W_1)}{\mu}-2+O(n^{-1})=\frac{k-1}{k}-2+O(n^{-1})$, so
    $$\begin{aligned}\rho-\mu&=-1+\frac{2w_2-\frac{k-1}{k}}{\mu+w_2}+\frac{w_2\frac{k-1}{k}}{(\mu+w_2)(\mu+w_2-1)}+\frac{w_3}{\mu+w_3}+O(n^{-2})\\
    &=-1+\frac{2w_2}{\mu+w_2}-\frac{2(k-1)}{3k}\frac{1}{n}+\frac{w_3}{\mu+w_3}+O(n^{-2}).\end{aligned}$$
    \uline{Case 1:} $w_3=w_2-1$. Using \autoref{spectral balance 4 equation 1} we obtain an initial estimate of $\mu$:
    $$\begin{aligned}\mu&=w_2+w_3+\frac{w_2(Y_1-w_2)}{\mu+w_2}+\frac{(w_2-1)(Y_1-w_2+1)}{\mu+w_3}\\
    &=2w_2-1+\frac{w_2(2e(W_1)/\mu-2)}{\mu+w_2}+\frac{w_2(2e(W_1)/\mu-1)}{\mu+w_2-1}+O(n^{-1})\\
    &=2w_2-1+\frac{w_2(4e(W_1)/\mu-3)}{\mu+w_2}+O(n^{-1})\\
    &=2w_2-2+\frac{2}{3}\frac{k-1}{k}+O(n^{-1}).\end{aligned}$$
    Also, note that
    $$\begin{aligned}
        \frac{w_3}{\mu+w_3}&=\frac{w_2-1}{\mu+w_2-1}=\frac{w_2}{\mu+w_2}+\frac{w_2}{(\mu+w_2)(\mu+w_2-1)}-\frac{1}{n}+O(n^{-2})\\
        &=\frac{w_2}{\mu+w_2}-\frac{2}{3n}+O(n^{-2}).
    \end{aligned}$$
Thus
$$\begin{aligned}
    \rho-\mu&=-1+\frac{\mu}{\mu+w_2}+\frac{w_2}{\mu+w_2}+\left(2-\frac{2}{3}\frac{k-1}{k}-\frac{2}{3}\frac{k-1}{k}-\frac{2}{3}\right)\frac{1}{n}+O(n^{-2})\\
    &=\frac{4}{3}\left(1-\frac{k-1}{k}\right)\frac{1}{n}+O(n^{-2})
\end{aligned}$$
and we are done.\\
\uline{Case 2:} $w_3=w_2$. Similarly to case 1 we obtain an initial estimate of $\mu$:
$$\begin{aligned}
    \mu&=w_2+w_3+\frac{w_2(Y_1-w_2)}{\mu+w_2}+\frac{w_2(Y_1-w_2)}{\mu+w_2}\\
    &=2w_2+\frac{2w_2(Y_1-w_2)}{\mu+w_2}\\
    &=2w_2-\frac{4}{3}+\frac{2}{3}\frac{k-1}{k}+O(n^{-1}).
\end{aligned}$$
Thus,
$$\begin{aligned}
\rho-\mu&=-1+\frac{\mu}{\mu+w_2}+\frac{w_2}{\mu+w_2}+\left(\frac{4}{3}-\frac{2}{3}\frac{k-1}{k}-\frac{2}{3}\frac{k-1}{k}\right)\frac{1}{n}+O(n^{-2})\\
&=\frac{4}{3}\left(1-\frac{k-1}{k}\right)\frac{1}{n}+O(n^{-2})
\end{aligned}$$
and we are done.
\end{proof}

Henceforth, whenever we need precise eigenvalue estimates we will follow the argument of \autoref{spectral balance 4}, so we will not repeat every detail of such arguments. Before completing our analysis of the part sizes of $G$, we need the following lemma.

\begin{lemma} \label{tree comparison lemma}
    Suppose $G_1$, $G_2$ are graphs obtained from $T_{n,r}$ by embedding the same number of trees of order $k\ge 4$ in a large part so as to cover all but $O(1)$ vertices in the part. Then
    $$\lambda(G_1)-\lambda(G_2)\le\left[\frac{k-5+6/k}{r-1}\right]\frac{1}{n}+O(n^{-2}).$$
    Moreover, if all the trees embedded in $G_1$ are stars and all the trees embedded in $G_2$ are paths, then equality holds.
\end{lemma}
\begin{proof}
    Let $G_{\mathrm{star}}$ and $G_{\mathrm{path}}$ be the graphs obtained from $G_1$ or $G_2$ by replacing each embedded tree with a star or path, respectively.
    \begin{claim} \label{tree comparison lemma claim 1}
    Let $K$ be any graph obtained by embedding a disjoint union of trees of order $k$ in one part of $T_{n,r}$, where one of the trees $T$ is not a star. Then there is another tree $T'$ with more pendant vertices than $T$ such that, if $K'$ is obtained from $K$ by replacing $T$ with $T'$, we have $\lambda(K')\ge\lambda(K)$.
    \end{claim}
    \begin{proof}
        Choose a non-pendant vertex $u$ as the root of $T$ and choose $v\in V(T)$ at maximum distance from $u$. Since $T$ is not a star, $d(v,u)\ge 2$ and also $v$ is a pendant vertex. Let $v\sim_T w$. Since $v$ is at maximum distance from $u$ and $w\ne u$, $w$ has a unique neighbor $w'$ such that $w'$ is not a pendant vertex. Let $T'$ be obtained by deleting the edge $wv'$ and adding the edge $w'v'$ for every pendant $v'\sim_Tw'$, and note that $T'$ has more pendant vertices than $T$. Now $K'$ is obtained from $K$ by performing a Kelmans transformation and so $\lambda(K')\ge\lambda(K)$ (see e.g. \cite{csikvari2009conjecture}).
    \end{proof}
    By repeatedly applying \autoref{tree comparison lemma claim 1} to $G_1$, it follows that $\lambda(G_1)\le\lambda(G_{\mathrm{star}})$.
    \begin{claim} \label{tree comparison lemma claim 2}
        Let $K$ be any graph obtained by embedding a disjoint union of trees of order $k$ in one part of $T_{n,r}$, where one of the trees $T$ is not a path. Then there is another tree $T'$ of greater diameter than $T$ such that, if $K'$ is obtained from $K$ by replacing $T$ with $T'$, we have $\lambda(K')\le\lambda(K)$.
    \end{claim}
    \begin{proof}
        Let a longest path $P$ in $T$ have endpoints $u$ and $u'$ (so that $u,u'$ are pendant in $T$), and note that if $T$ is not a path then there is a vertex $v$ of degree at least $3$ which is on $P$. Let $v\sim w$ with $w\not\in\{u,u'\}$. Let $T'$ be obtained from $T$ by deleting $vw$ and adding $uw$ and note that $T'$ is a tree of greater diameter than $T$. Let $\mu=\lambda(K)$, with eigenvector $y$, scaled so that the sum of $y_v$ over $v$ not in the large part equals $\mu$. For any $x\in V(T)$ one can show $y_x=1+\frac{d_T(x)}{\mu}+O(n^{-2})$. Therefore
        $$\begin{aligned}
            \frac{1}{2y_w}(y^TA(K')y-y^TA(K)y)&\le 1+\frac{1}{\mu}-1-\frac{3}{\mu}+O(n^{-2})<0.
        \end{aligned}$$
    \end{proof}
    By repeatedly applying \autoref{tree comparison lemma claim 2} to $G_2$, it follows that $\lambda(G_2)\ge\lambda(G_{\mathrm{path}})$.
    Let $V_1,\ldots,V_r$ be the common partition of $G_{\mathrm{star}}$ and $G_{\mathrm{path}}$, with $v_i:=|V_i|$. Let $v_1,\ldots,v_h=v_1$ and $v_{h+1},\ldots,v_r=v_1-1$. Let $\mu=\lambda(G_{\mathrm{star}})$, $\rho=\lambda(G_{\mathrm{path}})$, and let $y,z$ be the corresponding eigenvectors, scaled so that $y_1:=\frac{1}{\mu}\sum_{v\not\in V_1}y_v=1$ and $z_1:=\frac{1}{\rho}\sum_{v\not\in V_1}z_v=1$. Similarly to \autoref{spectral balance 4 claim 1}, we find $\mu-\rho=O(n^{-1})$.
    For $v\in V_1$, let $d_s(v):=d_{G_{\mathrm{star}}[V_1]}(v)$ and $d_p(v):=d_{G_{\mathrm{path}[V_1]}}(v)$, and let $d_s'(v)$ ($d_p'(v)$) be the number of walks of length 2 starting from $v$ in $G_{\mathrm{star}}[V_1]$ ($G_{\mathrm{path}}[V_1]$). The second-degree eigenvector equations give that for $v\in V_1$,
    $$\begin{aligned}
        y_v&=1+\frac{d_s(v)}{\mu}+\frac{d_s'(v)}{\mu^2}+O(n^{-3})\\
        z_v&=1+\frac{d_p(v)}{\mu}+\frac{d_p'(v)}{\mu^2}+O(n^{-3}).
    \end{aligned}$$
    Let $Y_1=\sum_{v\in V_1}y_v$, $Z_1=\sum_{v\in V_1}z_v$. We obtain
    $$\begin{aligned}
        Y_1&=v_1+\frac{2e(V_1)}{\mu}+\frac{1}{\mu^2}\sum_{v\in V_1}d_s'(v)+O(n^{-2})\\
        Z_1&=v_1+\frac{2e(V_1)}{\mu}+\frac{1}{\mu^2}\sum_{v\in V_1}d_p'(v)+O(n^{-2}).
    \end{aligned}$$
    Observe that the number of walks of length two in a path of order $k\ge 4$ is $4k-6$ and the number of walks of length two in $K_{1,k-1}$ is $k(k-1)$. Therefore,
    $$\begin{aligned}
        Y_1-Z_1&=\frac{1}{\mu^2}\sum_{v\in V_1}(d_s'(v)-d_p'(v))+O(n^{-2})\\
        &=\frac{1}{\mu^2}\left(\frac{v_1-O(1)}{k}k(k-1)-\frac{v_1-O(1)}{k}(4k-6)\right)+O(n^{-2})\\
        &=\frac{v_1(k-5+6/k)}{\mu^2}+O(n^{-2}).
    \end{aligned}$$
    Now for $v\in V_i$ $(i\ge 2)$, the eigenvector equations give
    $$y_v=\begin{cases}
        1+\frac{Y_1-v_1}{\mu+v_1},&2\le i\le h\\
        1+\frac{Y_1-v_1+1}{\mu+v_1}+\frac{Y_1-v_1+1}{(\mu+v_1)^2}+O(n^{-3}),&h+1\le i\le r
    \end{cases}$$
    and
    $$z_v=\begin{cases}
        1+\frac{Z_1-v_1}{\mu+v_1}+O(n^{-3}),&2\le i\le h\\
        1+\frac{Z_1-v_1+1}{\mu+v_1}+\frac{Z_1-v_1+1}{(\mu+v_1)^2}+O(n^{-3}),&h+1\le i\le r.
    \end{cases}$$
    We obtain
    $$\begin{aligned}
        \mu=&\,\mu y_1=(h-1)v_1\left(1+\frac{Y_1-v_1}{\mu+v_1}\right)\\
        &+(r-h)(v_1-1)\left(1+\frac{Y_1-v_1+1}{\mu+v_1}+\frac{Y_1-v_1+1}{(\mu+v_1)^2}\right)+O(n^{-2})\\
        \rho=&\,\rho z_1=(h-1)v_1\left(1+\frac{Z_1-v_1}{\mu+v_1}\right)\\
        &+(r-h)(v_1-1)\left(1+\frac{Z_1-v_1+1}{\mu+v_1}+\frac{Z_1-v_1+1}{(\mu+v_1)^2}\right)+O(n^{-2}).
    \end{aligned}$$
    Thus
    $$\begin{aligned}
        \mu-\rho&=(h-1)v_1\frac{Y_1-Z_1}{\mu+v_1}+(r-h)(v_1-1)\frac{Y_1-Z_1}{\mu+v_1}+(r-h)(v_1-1)\frac{Y_1-Z_1}{(\mu+v_1)^2}+O(n^{-2})\\
        &=\frac{(r-1)v_1(Y_1-Z_1)}{\mu+v_1}+O(n^{-2})\\
        &=\left[\frac{k-5+6/k}{r-1}\right]\frac{1}{n}+O(n^{-2}).
    \end{aligned}$$
    This proves the equality part of the result, and the inequality part follows from $\lambda(G_{\mathrm{path}})\le\lambda(G_1),\lambda(G_2)\le\lambda(G_{\mathrm{star}}).$
\end{proof}

\begin{lemma} \label{spectral balance 4.5}
    For each $i=2,\ldots,r$, we have $w_i\ge w_1-1$.
\end{lemma}
\begin{proof}
    Suppose not, and assume without loss of generality that $w_1=w_2+2$. Let $w_2,\ldots,w_h=w_1-2$ and $w_{h+1},\ldots,w_r=w_1-1$. Let $G'$ be obtained from $G$ by converting every vertex in $W^m$ to a $W$-ordinary vertex, and let $a=e(G')-e(G)=O(1)$. Let $G''$ be obtained from $G'$ by deleting a pendant edge in $G'[W_1]$. Then $e(G'')=e(G)+a-1$, and $G''$ has an isolated vertex $u'$ in $G''[W_1]$. Let $\mu=\lambda(G'')$, $\rho=\lambda(G_{12}'')$, where $G_{12}''$ is obtained by transferring $u'$ to $W_2$; then $G_{12}''$ is $\mathcal F$-free and $e(G_{12}'')=e(G'')+1=e(G)+a$. Let $y$ be the principal eigenvector of $G''$ and $z$ be the principal eigenvector of $G_{12}''$, scaled so that $y_1:=\frac{1}{\mu}\sum_{v\not\in W_1}y_v=1$ and $z_1:=\frac{1}{\rho}\sum_{v\not\in W_1-\{u'\}}z_v=1$. Similarly to \autoref{spectral balance 4 claim 1}, we find $\rho-\mu=O(n^{-1})$.
    \begin{claim} \label{spectral balance 4.5 claim 2}
        $\rho-\mu=\left[2-\frac{2}{r}+\frac{4(k-1)}{kr}\right]\frac{1}{n}+O(n^{-2}).$
    \end{claim}
    \begin{proof}
        For $v\in W_1$ the second-degree eigenvector equation gives
        $$y_v=1+\frac{d_{W_1}(v)}{\mu}+\frac{d_{W_1}'(v)}{\mu^2}+O(n^{-3})$$
        and for $v\in W_1-\{u'\}$, gives
        $$z_v=1+\frac{d_{W_1}(v)}{\mu}+\frac{d_{W_1}'(v)}{\mu^2}+O(n^{-3}).$$
        Let $Y_1=\sum_{v\in W_1}y_v$, $Z_1=\sum_{v\in W_1-\{u'\}}z_v$. Then we obtain
        $$\begin{aligned}
            Y_1&=w_1+\frac{2(e_G(W_1)-1)}{\mu}+\frac{\sum_{v\in W_1-\{u'\}}d_{W_1}'(v)}{\mu^2}+O(n^{-2})\\
            Z_1&=Y_1-1+O(n^{-2}).
        \end{aligned}$$
        In particular $Y_1=w_1+\frac{2(k-1)}{k(r-1)}+O(n^{-1})$. The first-degree eigenvector equations now give that if $v\in W_i$ with $i\ge 2$, we have
        $$y_v=\begin{cases}
            1+\frac{Y_1-w_2}{\mu+w_2},&2\le i\le h\\
            1+\frac{Y_1-w_2-1}{\mu+w_2}-\frac{Y_1-w_2-1}{(\mu+w_2)^2}+O(n^{-3}),&i>h
        \end{cases}$$
        and if $v\in W_2\cup\{u'\}$ or $v\in W_i$ for $i>2$, then
        $$z_v=\begin{cases}
            1+\frac{Y_1-w_2-1}{\mu+w_2}+O(n^{-3}),&3\le i\le h\\
            1+\frac{Y_1-w_2-2}{\mu+w_2}-\frac{Y_1-w_2-2}{(\mu+w_2)^2}+O(n^{-3}),&i=2\text{ or }i>h.
        \end{cases}$$
        Thus, $\mu=\mu y_1$ and $\rho=\rho z_1$ give
        $$\begin{aligned}
            \mu=&\,(h-1)w_2\left(1+\frac{Y_1-w_2}{\mu+w_2}\right)\\
            &+(r-h)(w_2+1)\left(1+\frac{Y_1-w_2-1}{\mu+w_2}-\frac{Y_1-w_2-1}{(\mu+w_2)^2}\right)+O(n^{-2})\\
            \rho=&\,(h-2)w_2\left(1+\frac{Y_1-w_2-1}{\mu+w_2}\right)\\
            &+(r-h+1)(w_2+1)\left(1+\frac{Y_1-w_2-2}{\mu+w_2}-\frac{Y_1-w_2-2}{(\mu+w_2)^2}\right)+O(n^{-2}).
        \end{aligned}$$
        We now obtain an initial estimate on $\mu$. We have
        $$\begin{aligned}
            \mu=&\,(h-1)w_2+(r-h)(w_2+1)+\frac{(h-1)w_2\left(w_1-w_2+\frac{2(k-1)}{k(r-1)}\right)+O(1)}{\mu+w_2}\\
            &+\frac{(r-h)(w_2+1)\left(w_1-w_2-1+\frac{2(k-1)}{k(r-1)}\right)+O(1)}{\mu+w_2}+O(n^{-1})\\
            =&\,(r-1)w_2+r-h+\frac{h}{r}-\frac{2}{r}+1+\frac{2(k-1)}{kr}+O(n^{-1}).
        \end{aligned}$$
        Now
        $$\begin{aligned}
            \rho-\mu=&\,-w_2+w_2\frac{(h-2)(Y_1-w_2-1)-(h-1)(Y_1-w_2)}{\mu+w_2}\\
            &+w_2+1+(w_2+1)\frac{(r-h+1)(Y_1-w_2-2)-(r-h)(Y_1-w_2-1)}{\mu+w_2}\\
            &-(w_2+1)\frac{(r-h+1)(Y_1-w_2-2)-(r-h)(Y_1-w_2-1)}{(\mu+w_2)^2}+O(n^{-2})\\
            =&\,1+\frac{w_2(w_2-Y_1-h+2)}{\mu+w_2}\\
            &+\frac{(w_2+1)(Y_1-w_2-r+h-2)}{\mu+w_2}\\
            &-\frac{(w_2+1)(Y_1-w_2-r+h-2)}{(\mu+w_2)^2}+O(n^{-2})\\
            =&\,1-\frac{rw_2}{\mu+w_2}+\frac{Y_1-w_2-r+h-2}{\mu+w_2}-\frac{w_2(Y_1-w_2-r+h-2)}{(\mu+w_2)^2}+O(n^{-2}).
        \end{aligned}$$
        We estimate each term. Using the initial estimates of $\mu$ and $Y_1$, we have
        $$\begin{aligned}
            \frac{rw_2}{\mu+w_2}&=1-\left[r-h+\frac{h}{r}-\frac{2}{r}+1+\frac{2(k-1)}{kr}\right]\frac{1}{n}+O(n^{-2})\\
            \frac{Y_1-w_2-r+h-2}{\mu+w_2}&=\left[\frac{2(k-1)}{k(r-1)}-r+h\right]\frac{1}{n}+O(n^{-2})\\
            \frac{w_2(Y_1-w_2-r+h-2)}{(\mu+w_2)^2}&=\left[\frac{2(k-1)}{kr(r-1)}-1+\frac{h}{r}\right]\frac{1}{n}+O(n^{-1}).
        \end{aligned}$$
        So,
        $$\begin{aligned}\rho-\mu&=\left[-\frac{2}{r}+1+\frac{2(k-1)}{kr}+\frac{2(k-1)}{k(r-1)}-\frac{2(k-1)}{kr(r-1)}+1\right]\frac{1}{n}+O(n^{-2})\\
        &=\left[2-\frac{2}{r}+\frac{4(k-1)}{kr}\right]\frac{1}{n}+O(n^{-2}).\end{aligned}$$
    \end{proof}
    Note that $G_{12}''$ is obtained from $T_{n,r}$ by embedding trees of order $k$ in a large part, with possible 1 smaller tree. By \autoref{structure lemma 1} and applying a vertex transfer if necessary, some graph $H\in\mathrm{EX}(n,\mathcal F)$ is obtained from $T_{n,r}$ by embedding trees of order $k$ in a large part, then making a bounded number of edge modifications. There is a graph $H'$ obtained from $H$ by a bounded number of edge modifications and a graph $G_{12}'''$ obtained from $G_{12}''$ by a bounded number of edge modifications, so that each graph is obtained from $T_{n,r}$ by embedding the same number of trees of order $k$ in the large part; let $b=e(H)-e(H')=O(1)$, $a'=e(G_{12}''')-e(G_{12}'')=O(1).$ Now if $G\in\mathrm{EX}(n,\mathcal F)$ we are done, so we may assume
    $b+a'+a=e(H)-e(G)\ge 1$.\\
    \uline{Case 1:} $k\in\{1,2,3\}$. Then there is only one tree of order $k$, so $H'$ can be obtained from $G_{12}'''$ by a bounded number of edge modifications. By \autoref{bounded modification lemma}, we have
    $$\begin{aligned}
        \lambda(H)-\lambda(G)=&\,\lambda(H)-\lambda(H')+\lambda(H')-\lambda(G_{12}''')+\lambda(G_{12}''')-\lambda(G_{12}'')\\
        &+\lambda(G_{12}'')-\lambda(G'')+\lambda(G'')-\lambda(G)\\
        =&\,\left[2b+0+2a'+2-\frac{2}{r}+\frac{4(k-1)}{kr}+2(a-1)\right]\frac{1}{n}+O(n^{-2})>0
    \end{aligned}$$
    which gives a contradiction.\\
    \uline{Case 2:} $k\ge 4$.
    By \autoref{tree comparison lemma}, we have $\lambda(H')-\lambda(G_{12}''')\ge-\left[\frac{k-5+6/k}{r-1}\right]\frac{1}{n}+O(n^{-2})$. Then by \autoref{bounded modification lemma} we have
    $$\begin{aligned}
        \lambda(H)-\lambda(G)=&\,\lambda(H)-\lambda(H')+\lambda(H')-\lambda(G_{12}''')+\lambda(G_{12}''')-\lambda(G_{12}'')\\
        &+\lambda(G_{12}'')-\lambda(G'')+\lambda(G'')-\lambda(G)\\
        \ge&\,\left[2b-\frac{k-5+6/k}{r-1}+2a'+2-\frac{2}{r}+\frac{4(k-1)}{kr}+2(a-1)\right]\frac{1}{n}+O(n^{-2})\\
        \ge&\,\left[-\frac{k-5+6/k}{r-1}+2-\frac{2}{r}+\frac{4(k-1)}{kr}\right]\frac{1}{n}+O(n^{-2})
    \end{aligned}$$
    Now, the assumption $Q<c_1(r)$ and $1\le k\le1/(1-Qr)$ imply that $2-\frac{k-5+6/k}{r-1}-\frac{4(k-1)}{kr}>0$. Since $k>4/3$, we have $\frac{4(k-1)}{kr}-\frac{2}{r}>-\frac{4(k-1)}{kr}$ and so the bracketed expression above is positive, giving a contradiction.
\end{proof}

This establishes that $|w_i-w_j|\le 1$ for every $i,j$. We now complete the proof of \autoref{main theorem}. 

\begin{proof}[Proof of \autoref{main theorem}] As stated above, the assumption $Q<c_1(r)$ gives $2-\frac{k-5+6/k}{r-1}-\frac{4(k-1)}{kr}>0$, and we will use this estimate below. Let $H\in\mathrm{EX}(n,\mathcal F)$, with partition $V_1\sqcup\cdots \sqcup V_r$ according to \autoref{structure lemma 1}, and $v_i:=|V_i|$. If $v_1<v_i$ for some $i$ then $e(H_{i1})=e(H)$, so we may assume without loss of generality that $v_1\ge v_i$ for every $i$. Then either $w_1=v_1$ or $w_1=v_1-1$. Assume for a contradiction that $e(H)>e(G)$.\\
\uline{Case 1:} $w_1=v_1$. \\
\uline{Subcase 1.1:} $k\in\{1,2,3\}$. Then $H$ can be obtained from $G$ by a bounded number of edge deletions and additions, and by \autoref{bounded modification lemma} we are done.\\
\uline{Subcase 1.2:} $k\ge 4$. There is a graph $H'$ be obtained from $H$ by a bounded number of edge modifications and a graph $G'$ be obtained from $G$ by a bounded number of edge modifications such that $H'$ and $G'$ can be obtained from $T_{n,r}$ by embedding the same number of trees of order $k$ in the large part. Let $b=e(H)-e(H')=O(1)$, $a=e(G)-e(G')<b$. Then \autoref{tree comparison lemma} gives $\lambda(H')\ge\lambda(G')-\left[\frac{k-5+6/k}{r-1}\right]\frac{1}{n}+O(n^{-2})$. Hence, using \autoref{bounded modification lemma}
$$\begin{aligned}
    \lambda(H)-\lambda(G)&=\lambda(H)-\lambda(H')+\lambda(H')-\lambda(G')+\lambda(G')-\lambda(G)\\
    &\ge\left[2b-\frac{k-5+6/k}{r-1}-2a\right]\frac{1}{n}+O(n^{-2})\\
    &\ge\left[2-\frac{k-5+6/k}{r-1}\right]\frac{1}{n}+O(n^{-2}).
\end{aligned}$$
Since $2-\frac{k-5+6/k}{r-1}>0$, this gives a contradiction.\\
\uline{Case 2: $w_1=v_1-1.$} Then without loss of generality, $w_2=w_1+1=v_1$. Let $G'$ be obtained from $G$ by replacing every vertex in $W^m$ with a $W$-ordinary vertex. Let $a=e(G')-e(G)=O(1)$.
\begin{claim} \label{proof claim 1}
    We have $\lambda(G_{21}')-\lambda(G')=-\left[\frac{4(k-1)}{kr}\right]\frac{1}{n}+O(n^{-2}).$
\end{claim}
\begin{proof}
    Let $u'$ be the $W$-ordinary vertex in $W_2$ which was transferred to obtain $G_{21}'$. Let $\mu=\lambda(G'),\rho=\lambda(G_{21}')$, with corresponding eigenvectors $y,z$, scaled so that $y_1:=\frac{1}{\mu}\sum_{v\not\in W_1}y_v=1$ and $z_1:=\frac{1}{\rho}\sum_{v\not\in W_1\cup\{u'\}}z_v=1$. Let $w_2,\ldots,w_h=w_1+1$ and $w_{h+1},\ldots,w_r=w_1$. Similarly to \autoref{spectral balance 4.5}, we obtain the following preliminary estimates. First, $\rho-\mu=O(n^{-1})$. For $v\in W_1$,
    $$\begin{aligned}y_v&=1+\frac{d_{W_1}(v)}{\mu}+\frac{d_{W_1}'(v)}{\mu^2}+O(n^{-3})\\
    z_v&=1+\frac{d_{W_1}(v)}{\mu}+\frac{d_{W_1}'(v)}{\mu^2}+O(n^{-3}).\end{aligned}$$
    Also $z_{u'}=1$. Let $Y_1=\sum_{v\in W_1}y_v,$ $Z_1=\sum_{W_1\cup\{u'\}}z_v$. Then
    $$\begin{aligned}
        Y_1&=w_1+\frac{2e(W_1)}{\mu}+\frac{\sum_{v\in W_1}d_{W_1}'(v)}{\mu^2}+O(n^2)\\
        &=w_1+\frac{2(k-1)}{k(r-1)}+O(n^{-1})\\
        Z_1&=Y_1+1+O(n^{-2}).
    \end{aligned}$$
    If $v\in W_i$ for $i\ge 2$, then
    $$y_v=\begin{cases}
        1+\frac{Y_1-w_1-1}{\mu+w_1}-\frac{Y_1-w_1-1}{(\mu+w_1)^2}+O(n^{-3}),&2\le i\le h\\
        1+\frac{Y_1-w_1}{\mu+w_1},&h+1\le i\le r.
    \end{cases}$$
    For $v\in W_2-\{u'\}$ or $v\in W_i$ with $i\ge 3$,
    $$z_v=\begin{cases}
        1+\frac{Y_1-w_1}{\mu+w_1}-\frac{Y_1-w_1}{(\mu+w_1)^2}+O(n^{-3}),&3\le i\le h\\
        1+\frac{Y_1-w_1+1}{\mu+w_1}+O(n^{-3}),&i=2\text{ or }h+1\le i\le r.
    \end{cases}$$
    So,
    $$\begin{aligned}
        \mu=&\,(h-1)(w_1+1)\left(1+\frac{Y_1-w_1-1}{\mu+w_1}-\frac{Y_1-w_1-1}{(\mu+w_1)^2}\right)\\
        &+(r-h)w_1\left(1+\frac{Y_1-w_1}{\mu+w_1}\right)+O(n^{-2})\\
        \rho=&\,(h-2)(w_1+1)\left(1+\frac{Y_1-w_1}{\mu+w_1}-\frac{Y_1-w_1}{(\mu+w_1)^2}\right)\\
        &+(r-h+1)w_1\left(1+\frac{Y_1-w_1+1}{\mu+w_1}\right)+O(n^{-2}).
    \end{aligned}$$
    We obtain the initial estimate
    $$\mu=(r-1)w_1+h-1+\frac{2(k-1)}{kr}-\frac{h}{r}+\frac{1}{r}+O(n^{-1}).$$
    Then,
    $$\begin{aligned}
        \rho-\mu=&\,-1+\frac{(w_1+1)(w_1-Y_1+h-1)}{\mu+w_1}\\
        &-\frac{(w_1+1)(w_1-Y_1+h-1)}{(\mu+w_1)^2}+\frac{w_1(Y_1-w_1+r-h+1)}{\mu+w_1}+O(n^{-2})\\
        =&\,-1+\frac{rw_1}{\mu+w_1}+\frac{w_1-Y_1+h-1}{\mu+w_1}-\frac{w_1(w_1-Y_1+h-1)}{(\mu+w_1)^2}+O(n^{-2}).
    \end{aligned}$$
    Now 
    $$\begin{aligned}
        \frac{rw_1}{\mu+w_1}&=1-\left[h-1+\frac{2(k-1)}{kr}-\frac{h}{r}+\frac{1}{r}\right]\frac{1}{n}+O(n^{-2})\\
        \frac{w_1-Y_1+h-1}{\mu+w_1}&=\left[-\frac{2(k-1)}{k(r-1)}+h-1\right]\frac{1}{n}+O(n^{-2})\\
        \frac{w_1(w_1-Y_1+h-1)}{(\mu+w_1)^2}&=\left[-\frac{2(k-1)}{kr(r-1)}+\frac{h}{r}-\frac{1}{r}\right]\frac{1}{n}+O(n^{-2}).
    \end{aligned}$$
    Thus,
    $$\rho-\mu=-\left[\frac{4(k-1)}{kr}\right]\frac{1}{n}+O(n^{-2}).$$
\end{proof}
\uline{Subcase 2.1:} $k\in\{1,2,3\}$. Then $H$ can be obtained from $G_{21}'$ by a bounded number of edge modifications. Let $b=e(H)-e(G_{21}')=e(H)-e(G')=e(H)-e(G)-a$, so $b+a\ge1$. Then by \autoref{bounded modification lemma} and \autoref{proof claim 1}, we have
$$\begin{aligned}
    \lambda(H)-\lambda(G)&=\lambda(H)-\lambda(G_{21}')+\lambda(G_{21}')-\lambda(G')+\lambda(G')-\lambda(G)\\
    &\ge\left[2b-\frac{4(k-1)}{kr}+2a\right]\frac{1}{n}+O(n^{-2})>0,
\end{aligned}$$
a contradiction.\\
\uline{Subcase 2.2:} $k\ge 4$. Let $H'$ be obtained from $H$ by deleting some $b=O(1)$ edges, and $G_{21}''$ be obtained from $G_{21}'$ by deleting some $a'=O(1)$ edges, so that $H'$ and $G_{21}''$ are each obtained by embedding the same number of trees of order $k$ in a large part of $T_{n,r}$. Note that $e(H)>e(G)$ gives $b+a-a'\ge 1$. Then by \autoref{bounded modification lemma}, \autoref{tree comparison lemma}, and \autoref{proof claim 1}, we have
$$\begin{aligned}
    \lambda(H)-\lambda(G)=&\,\lambda(H)-\lambda(H')+\lambda(H')-\lambda(G_{21}'')+\lambda(G_{21}'')-\lambda(G_{21}')\\
    &+\lambda(G_{21}')-\lambda(G')+\lambda(G')-\lambda(G)\\
    &\ge\left[2b-\frac{k-5+6/k}{r-1}-2a'-\frac{4(k-1)}{kr}+2a\right]\frac{1}{n}+O(n^{-2})\\
    &\ge\left[2-\frac{k-5+6/k}{r-1}-\frac{4(k-1)}{kr}\right]\frac{1}{n}+O(n^{-2}).
\end{aligned}$$
The bracketed expression is positive, which gives a contradiction.
\end{proof}

Observe that in Subcase 2.2 of the proof of \autoref{main theorem}, if $e(H)=e(G)+1$, $k\ge 4$, and all the trees embedded in $H$ are paths while all the trees embedded in $G$ are stars, then according to \autoref{tree comparison lemma}, the equality
$$\lambda(H)-\lambda(G)=\left[2-\frac{k-5+6/k}{r-1}-\frac{4(k-1)}{kr}\right]\frac{1}{n}+O(n^{-2})$$
holds. In the next section we will show that there is a forbidden family $\mathcal F$ such that some $H\in\mathrm{EX}(n,\mathcal F)$ and some $\mathcal F$-free graph $G$ in fact satisfy these conditions, and then use this to prove \autoref{counterexample 1}.

\section{Counterexamples}

\subsection{Large extremal number} \label{section large extremal number}
\begin{proof}[Proof of \autoref{counterexample 1}]
    Suppose that $r\ge 3$, $n\equiv 1\pmod r$ is large enough, and $\lfloor n/r\rfloor=(n-1)/r\equiv 0\pmod k$ (note there are infinitely many such $n$). Let $G$ be obtained by embedding $U(K_{1,k-1},\lfloor n/r\rfloor)$ in a small part of $T_{n,r}$ and let $H$ be obtained by embedding $U(P_k,\lceil n/r\rceil)$ in a large part of $T_{n,r}$ and then extending one of the paths by a single edge within the part, so that $e(H)=e(G)+1$. Let $V_1\sqcup\cdots\sqcup V_r$ and $U_1\sqcup\cdots\sqcup U_r$ be the corresponding partitions of $V(G)$ and $V(H)$, respectively, so that the trees are embedded in $V_1$ and $U_1$. Since $k>\alpha$ we have that
    $$\frac{1}{r-1}k+\frac{4}{r}-\frac{5}{r-1}+\left(\frac{6}{r-1}-\frac{4}{r}\right)k^{-1}>2$$
    and in particular $k\ge 4$. Thus, by the observation below the proof of \autoref{main theorem}, this implies that $\lambda(H)<\lambda(G)$. So, to prove \autoref{counterexample 1} it suffices to construct a family $\mathcal F$ such that $\mathrm{EX}(n,\mathcal F)=\{H\}$ and $G$ is $\mathcal F$-free. Let $m>\max\{4,k-1\}$ and let $\mathcal F$ be the union of the following graph families.
    \begin{itemize}
        \item[(1)] $T_k+T_{m(r-1),r-1}$ for every tree $T_k$ on $k$ vertices for which $T_k\not\simeq K_{1,k-1},P_k$
        \item[(2)] $T_{k+1}+T_{m(r-1),r-1}$ for every tree $T_{k+1}$ on $k+1$ vertices for which $T_{k+1}\not\simeq P_{k+1}$
        \item[(3)] $(2\cdot P_{k+1})+T_{m(r-1),r-1}$
        \item[(4)] $(K_{1,k-1}\cup P_{k+1})+T_{m(r-1),r-1}$
        \item[(5)] $C_\ell+T_{m(r-1),r-1}$ for $3\le\ell\le k+1$
        \item[(6)] $K_{r+2}$
    \end{itemize}
    \begin{claim}
        $H$ is $\mathcal F$-free.
    \end{claim}
    \begin{proof}
        Consider any copy $K$ of $T_{m(r-1),r-1}$ which appears in $H$ and assume for a contradiction that a forbidden subgraph of the form $F+K$ appears. If $V(K)\subseteq V-U_1$, then $N(V(K))=U_1$, so no forbidden graph of the form $F+K$ appears, by the definition of $H[U_1]$. Otherwise, $V(K)\ni v\in U_1$. If $|V(K)\cap U_1|\ge 3$, then $N(V(K)\cap U_1)\cap U_1=\emptyset$, hence $V(F)\cap U_1=\emptyset$. Then $V(F)$ intersects two different parts $U_2,U_3$, implying that $K$ is contained in the $r-2$ parts $U_1,U_4,\ldots,U_r$. But one can check $|V(K)\cap U_i|\le m$ since $\Delta(H[U_i])\le 2< m$ so $|V(K)|\le(r-2)m$, a contradiction. Thus, $|V(K)\cap U_1|\le 2$, so $|V(K)-U_1|\ge m(r-1)-2$, and thus $V(K)\cap U_i\ne\emptyset$ for each $i$ as $m(r-2)+2<m(r-1)-2$. Hence, $N(V(K))\subseteq U_1$, and the previous argument applies. This covers forbidden subgraphs of forms (1) through (5). Finally, if $H$ contains a copy $K$ of $K_{r+2}$, then $|V(K)\cap U_i|\le 1$ for each $i>1$ and $|V(K)\cap U_1|\le 2$. Hence $|V(K)|\le r+1$, a contradiction.
    \end{proof}
    \begin{claim}
        $G$ is $\mathcal F$-free.
    \end{claim}
    \begin{proof}
        Consider any copy $K$ of $T_{m(r-1),r-1}$ which appears in $G$. If $V(K)\subseteq V-V_1$, then $N(V(K))=V_1$, so no forbidden graph of the form $F+K$ appears, by the definition of $G[V_1]$. Otherwise, $V(K)\ni v\in V_1$. If $|V(K)\cap V_1|\ge k$, then $N(V(K)\cap V_1)\cap V_1=\emptyset$, hence $V(F)\cap V_1=\emptyset$. Then $V(F)$ intersects two different parts $V_2,V_3$, implying that $K$ is contained in the $r-2$ parts $V_1,V_4,\ldots,V_r$. But one can check $|V(K)\cap V_i|\le m$ since $\Delta(G[V_i])\le k-1< m$ so $|V(K)|\le (r-2)m$, a contradiction. Thus, $|V(K)\cap V_1|\le k-1$, so $|V(K)-V_1|\ge m(r-1)-k+1$, and thus $V(K)\cap V_i\ne\emptyset$ for each $i$ since $m\ge k-1$. Hence $N(V(K))\subseteq V_1$, and the previous argument applies. This covers forbidden subgraphs of forms (1) through (5). Finally, if $G$ contains a copy $K$ of $K_{r+2}$, then $|V(K)\cap U_i|\le 1$ for each $i>1$ and $|V(K)\cap V_1|\le 2$. Hence $|V(K)|\le r+1$, a contradiction.
    \end{proof}
    \begin{claim} \label{claim extremal graph}
        If $n$ is large enough we have $\mathrm{EX}(n,\mathcal F)=\{H\}$.
    \end{claim}
    \begin{proof}
        Let $H'\in\mathrm{EX}(n,\mathcal F)$. Following \autoref{section structure of edge extremal} through \autoref{balance 2} shows that $H'$ is obtained from $K_{W_1,\ldots,W_r}$ by embedding trees in $W_1$, where $|w_i-w_j|\le 1$ for $i,j>1$ and $w_1-2\le w_i\le w_1+1$ for $i>1$. Suppose there are $h$ parts of size $w_1-2$ and $r-h-1$ parts of size $w_1-1$, so that
        $$\begin{aligned}
            w_1+h(w_1-2)+(r-h-1)(w_1-1)=n\Longrightarrow
            w_1=\frac{n-1}{r}+1+\frac{h}{r}.
        \end{aligned}$$
        Now $n\equiv 1\pmod r$, the right-hand side is an integer, and $0\le h\le r-1$, so we have $h=0$, and in fact $K_{W_1,\ldots,W_r}=T_{n,r}$. By (1) through (5), $H'[W_1]$ is a disjoint union of trees, each of which is $P_{k+1}$, $P_k$, $K_{1,{k-1}}$, or of order smaller than $k$. Moreover, at most one tree is $P_{k+1}$ and such a tree can only exist if there is no $K_{1,k-1}$. From these restrictions and the fact that $\lfloor n/r\rfloor\equiv 0\pmod k$ it is clear that $e(H'[W_1])$ (hence $e(H')$) is maximized exactly when $H'\simeq H$.
    \end{proof}
    By the discussion above, this completes the proof of \autoref{counterexample 1}.
\end{proof}

\subsection{Chromatic number 3} \label{section chromatic number 3}

\begin{proof}[Proof of \autoref{counterexample 2}]
Let $\mathcal F$ consist of $K_4$ and all graphs of the form $T_4+(m\cdot K_1)$, where $T_4$ is a tree on four vertices. (Here $m$ is a sufficiently large constant.) Now $\chi(F)=3$ and $\mathcal F$ satisfies the assumptions of \autoref{structure lemma 1} with $Q=\frac{1}{3}$. Let $n=2p$, where $p\equiv 1\pmod 3$. Let $G$ be the graph obtained by embedding $(p-1)/3\cdot P_3$ in the smaller part of $K_{p-1,p+1}$; let $H$ be the graph obtained by embedding $(p-1)/3\cdot P_3$ in one part of $K_{p,p}$, let $H'$ be the graph obtained by embedding $((p-4)/3\cdot P_3)\cup (2\cdot P_2)$ in one part of $K_{p,p}$, and let $H''$ be the graph obtained by embedding $((p-1)/3\cdot P_3)\cup P_2$ in the large part of $K_{p+1,p-1}$. Similarly to \autoref{claim extremal graph} one can prove $\mathrm{EX}(n,\mathcal F)=\{H,H',H''\}$. Now $G,H,H',H''$ have equitable partitions with quotient matrices
$$\begin{aligned}
    B=\begin{pmatrix}
    0&2&p+1\\
    1&0&p+1\\
    \frac{p-1}{3}&\frac{2(p-1)}{3}&0
\end{pmatrix},&\ \ C=\begin{pmatrix}
    0&2&0&p\\
    1&0&0&p\\
    0&0&0&p\\
    \frac{p-1}{3}&\frac{2(p-1)}{3}&1&0
\end{pmatrix},\\ C'=\begin{pmatrix}
    0&2&0&p\\
    1&0&0&p\\
    0&0&1&p\\
    \frac{p-4}{3}&\frac{2(p-4)}{3}&4&0
\end{pmatrix},&\ \ C''=\begin{pmatrix}
    0&2&0&p\\
    1&0&0&p\\
    0&0&1&p\\
    \frac{p-4}{3}&\frac{2(p-4)}{3}&4&0
\end{pmatrix},
\end{aligned}$$
respectively. We find that
$\lambda(H),\lambda(H'),\lambda(H'')<p+\frac{2}{3}-\frac{1}{5p}<\lambda(G)$ by substituting this value into the characteristic polynomials. Since $G\not\in\mathrm{EX}(n,\mathcal F)$, this implies that $\mathrm{SPEX}(n,\mathcal F)\cap\mathrm{EX}(n,\mathcal F)=\emptyset$.
\end{proof}

\section{Proof of \autoref{corollary}}
    First we show that $\alpha$ is not an integer. We have
    $$\begin{aligned}\alpha&=\frac{r-1}{2}\left(2+\frac{5}{r-1}-\frac{4}{r}+\sqrt{\left(2+\frac{5}{r-1}-\frac{4}{r}\right)^2-\frac{4}{r-1}\left(\frac{6}{r-1}-\frac{4}{r}\right)}\right)\\
    &=r-1+\frac{5}{2}-2+\frac{2}{r}+\sqrt{\left(2+\frac{5}{r-1}-\frac{4}{r}\right)^2\frac{(r-1)^2}{4}-(r-1)\left(\frac{6}{r-1}-\frac{4}{r}\right)}\\
    &=\frac{2r(r-1)+5r-4r+4}{2r}+\sqrt{\left(\frac{2r^2-r+2}{2r}\right)^2-\frac{5}{r}+\frac{3}{r^2}}.\end{aligned}$$
    For $i\in\mathbb Z$ let $s_i=(i/2r)^2$ and let $S=\{s_i:i\in\mathbb Z\}$. Note that $s_{2r^2-r+2}-s_{2r^2-r+1}=1-\frac{1}{2r}-\frac{3}{4r^2}.$ Whenever $r\ge 6$, we have
    $$0<\frac{5}{r}+\frac{3}{r^2}<1-\frac{1}{2r}-\frac{3}{4r^2}$$
    so the expression in the square root above is not an element of $S$ and therefore $\alpha$ is not an integer. When $r=3,4,5$, one may check directly that $\alpha$ is not an integer.

    Now let $k=\lceil\alpha\rceil>\alpha$. \autoref{counterexample 1} gives that $c(r)\le\frac{k-1}{kr}$. Suppose $Q<\frac{k-1}{kr}$. Let $\mathcal F$ be any finite family of graphs with $\chi(\mathcal F)=r+1$ and $\mathrm{ex}(n,\mathcal F)\le e(T_{n,r})+Qn$ for large enough $n$. Then we have
    $$\begin{aligned} k'&:=k(\mathcal F)\le\left\lfloor\frac{1}{1-Qr}\right\rfloor\le k-1\\
    &<\frac{r-1}{2}\left(2+\frac{5}{r-1}-\frac{4}{r}+\sqrt{\left(2+\frac{5}{r-1}-\frac{4}{r}\right)^2-\frac{4}{r-1}\left(\frac{6}{r-1}-\frac{4}{r}\right)}\right)\end{aligned}$$
    and so $\frac{k'-1}{k'r}<c_1(r)$. Moreover, \autoref{structure lemma 1} gives that $\mathrm{ex}(n,\mathcal F)\le e(T_{n,r})+\frac{k'-1}{k'r}n+O(1)$ for $n$ large enough, as $k'=k(\mathcal F)$. Hence, we may apply \autoref{main theorem} to show that $\mathrm{SPEX}(n,\mathcal F)\subseteq\mathrm{EX}(n,\mathcal F)$ for $n$ large enough. Therefore, $c(r)\ge Q$.\qed

\section{Concluding remarks}

\autoref{main theorem} gives a sharp condition on the value of $\mathrm{ex}(n,\mathcal F)$ which guarantees that $\mathrm{SPEX}(n,\mathcal F)\subseteq\mathrm{EX}(n,\mathcal F)$, when $\chi(\mathcal F)=r+1\ge 4$. As shown by \autoref{counterexample 2}, \autoref{main theorem} does not hold if we allow $r=2$ and $k(\mathcal F)=3$. We attempted to provide a similar counterexample in the case of $r=2,k(\mathcal F)=2$ but were unable to do so because $\lambda(H)$ is in fact greater than $\lambda(G)$ by a term of $\Theta(n^{-2})$, where $G$ and $H$ are graphs defined analogously to the ones in \autoref{section chromatic number 3}. This fact, combined with the result of \cite{wang2023conjecture}, suggested \autoref{conjecture}. A solution to \autoref{conjecture} would be interesting as it potentially requires stronger techniques than those of this paper and those of \cite{wang2023conjecture}.

We believe that the methods of this paper are general enough to apply to many specific families $\mathcal F$ which do not satisfy the assumptions of \autoref{main theorem}. Even for $\mathcal F$ such that $\mathrm{SPEX}(n,\mathcal F)\not\subseteq\mathrm{EX}(n,\mathcal F)$, large parts of the proof of \autoref{main theorem} go through and may complete much of the work in determining the extremal graphs.

Besides \autoref{conjecture}, the open questions we are most interested in are the following.
\begin{itemize}
    \item[(1)] Can the assumption that $\mathcal F$ is finite be dropped from \autoref{main theorem}?
    \item[(2)] What is the smallest value of $|\mathcal F|$ for which an example proving \autoref{counterexample 1} can be constructed?
\end{itemize}

\section{Acknowledgements}

The author would like to thank Michael Tait for suggesting the problem and for useful discussions which improved the quality of the paper. This research was partially supported by NSF DSM-2245556.

\printbibliography

\end{document}